\pdfoutput=1 

\documentclass[12pt]{extarticle}
\usepackage[utf8]{inputenc}
\usepackage[T1]{fontenc}

\usepackage{caption}
\oddsidemargin \evensidemargin
\usepackage{float}
\usepackage{resizegather}
\usepackage{amssymb}
\usepackage{pmboxdraw}
\usepackage{amsmath}
\usepackage{bbm}
\usepackage{amsthm}
\usepackage{amssymb}
\usepackage{dsfont}
\usepackage[dvipsnames]{xcolor}
\usepackage{graphicx}
\usepackage{alphabeta}
\usepackage{tikz,pgfplots}
\pgfplotsset{compat=1.18}
\usepackage{pgfplots}
\usepackage{comment}
\usepackage[margin=0.5cm]{caption}
\usepackage{hyperref}
\usepackage{subcaption}
\usepackage[all]{xy}
\usepackage[greek,english]{babel}
\usepackage[export]{adjustbox}
\usepackage{hyperref}
\usepackage{mathrsfs}  
\hypersetup{
    colorlinks=true,
    linkcolor=orange!80!black,
    filecolor=magenta,      
    urlcolor=RoyalBlue,
    citecolor=RoyalBlue,
}
\usepackage{enumitem}
\usepackage{booktabs}

\definecolor{myblue}{rgb}{0.00000,0.44700,0.74100}
\definecolor{myorange}{rgb}{0.8500, 0.3250, 0.0980}
\definecolor{myyellow}{rgb}{0.9290, 0.6940, 0.1250}
\definecolor{mypurple}{rgb}{0.4940, 0.1840, 0.5560}
\definecolor{mygreen}{rgb}{0.4660, 0.6740, 0.1880}

\usepackage[activate={true,nocompatibility},final,tracking=true,kerning=true,spacing=true,factor=1100,stretch=10,shrink=10]{microtype}

\let\OLDthebibliography\thebibliography
\renewcommand\thebibliography[1]{
  \OLDthebibliography{#1}
  \setlength{\parskip}{0.4pt}
  \setlength{\itemsep}{5pt plus 0.3ex}
}

\usepackage{listings}

\usepackage{algorithm}
\usepackage{algpseudocode}

\newtheorem{theorem}{Theorem}[section]
\newtheorem{theorem*}[theorem]{Theorem*}

\newtheorem{lemma}[theorem]{Lemma}

\newtheorem{proposition}[theorem]{Proposition}

\newtheorem{thm/conj}[theorem]{Theorem/Conjecture}

\theoremstyle{definition}
\newtheorem{examplex}[theorem]{Example}
\newenvironment{example}
{\pushQED{\qed}\examplex}
{\popQED\endexamplex}
\newtheorem{definition}[theorem]{Definition}
\newtheorem{remark}[theorem]{Remark}

\theoremstyle{remark}

\renewcommand{\P}{\mathbb{P}}
\newcommand{\C}{\mathbb{C}}
\newcommand{\qq}{\mathbb{Q}}

\renewcommand{\L}{\mathcal{L}}
\renewcommand{\O}{\mathcal{O}}
\newcommand{\V}{V}
\newcommand{\X}{\mathcal{X}}

\DeclareMathOperator{\mult}{mult}
\DeclareMathOperator{\sing}{Sing}

\newcommand{\midcolon}{\, : \,}

\newcommand\restr[2]{{
  \left.\kern-\nulldelimiterspace 
  #1
  \vphantom{\big|} 
  \right|_{#2}
  }}

\title{\bf Euler Stratifications of Hypersurface Families}
\author{Simon Telen and Maximilian Wiesmann}
\date{}

\begin{document}

\maketitle

\begin{abstract}
    \noindent We stratify families of projective and very affine hypersurfaces according to their topological Euler characteristic. Our new algorithms compute all strata using algebro-geometric techniques. For very affine hypersurfaces, we investigate and exploit the relation to critical point computations. Euler stratifications are relevant in particle physics and algebraic statistics. They fully describe the dependence of the number of master integrals, respectively the maximum likelihood degree, on kinematic or model parameters.
\end{abstract}

\section{Introduction}

A nonzero homogeneous polynomial $F \in \mathbb{C}[x_0, \ldots, x_d]$ of degree $n$ defines an algebraic hypersurface $V(F) \subset \mathbb{P}^d$. We study the dependence of the (topological) Euler characteristic of $V(F)$ on the coefficients of $F$. More precisely, we consider a family $V(F(x,z)) \subset \mathbb{P}^d \times Z$ of hypersurfaces over an irreducible variety $Z$.
We seek to compute an explicit description of the loci in $Z$ on which the hypersurface defined by $F$ has a prescribed Euler characteristic.
An Euler stratification decomposes $Z$ into such loci. For a precise definition, see Section \ref{sec:2}.

We must clarify what we mean by an ``explicit description'' of these loci, or of the strata in an Euler stratification. It turns out that Euler strata are constructible, meaning essentially that they can be described by polynomials. We illustrate this for plane conics. 

\begin{example}[$d = n = 2$] \label{ex:conics_intro}
    Let $Z = \mathbb{P}^5$ and consider a generic ternary quadric
    \begin{equation} \label{eq:Fconics} F(x_0,x_1,x_2,z) \, = \, z_0 x_0^2 + z_1 x_0x_1 + z_2 x_0x_2 + z_3 x_1^2 + z_4 x_1x_2 + z_5 x_2^2\, .\end{equation}
    We emphasize that the $x_i$ are thought of as variables, and the $z_j$ are parameters. 
    The Euler characteristic of the curve $V(F(x,z))$ is either two or three. Indeed, if $V(F(x,z))$ is smooth and $F$ is irreducible then the Euler characteristic is $\chi(V(F(x,z))) = 2$. The same is true when $F = L^2$ is the square of a linear form $L$. If $V(F(x,z))$ is singular, then it is the union of two distinct lines, and it has Euler characteristic three. Algebraically, we have 
    \begin{equation} \label{eq:Mz} \chi(V(F(x,z))) = 3 \,\, \,  \Longleftrightarrow \, \,\, {\rm rank}\,  M(z) \, = \, 2, \quad \text{where} \quad M(z) \, = \, \begin{pmatrix}
        2z_0 & z_1 & z_2 \\ z_1 & 2z_3 & z_4\\
        z_2 & z_4 & 2z_5
    \end{pmatrix} .
    \end{equation}
    This defines a constructible subset of $\mathbb{P}^5$, consisting of the smooth points of the hypersurface 
    \begin{equation} \label{eq:conicdiscriminant} \nabla  \, = \, \{ z\in Z \, : \, \Delta \, = \, 0\}, \quad \text{where} \quad \Delta \, = \,  8\, z_0z_3z_5 - 2 \, z_0z_4^2 - 2 \, z_1^2z_5 + 2 \, z_1z_2z_4 - 2\, z_2^2z_3.  \end{equation}
    The singular locus $\nabla_{\rm sing}$ is cut out by the $2 \times 2$-minors of the matrix $M(z)$ in \eqref{eq:Mz}.
\end{example}

Euler strata in the parameter space $Z$ are closely related to the equisingular loci of our hypersurface families, i.e., the sets of parameter values for which the singularities of $V(F(x,z))$ are of the same type, in some appropriate sense. These loci are not well understood in general. For points in $\mathbb{P}^1$ $(d = 1)$, the problem comes down to computing coincident root loci \cite{chipalkatti2003equations}. Computations up to degree $n = 7$ are reported in \cite{lee2016duality}. For plane curves, the most well studied equisingular loci are Severi varieties. These are loci of nodal curves of fixed degree with a fixed number of nodes \cite{Fulton}.  For tropical~approaches to Severi varieties, see for instance \cite{dickenstein2017arithmetics, yang2013tropical}. Severi varieties for surfaces in $\mathbb{P}^3$ were~studied~in~\cite{Chiantini1998OnTS}. In 1988, Diaz and Harris wrote that the~equisingular stratification of plane curves of a fixed degree was out of reach \cite[p.~1]{diaz1988geometry}. Our paper addresses the problem using modern tools from computational algebraic geometry. We elaborate more on equisingular loci in a paragraph at the end of Section \ref{sec:2}.

A special case which deserves extra attention is when $F$ factors as 
\begin{equation}  \label{eq:factorscase}
F(x,z) = x_0\cdots x_d \, f(x,z). 
\end{equation} 
Let $T \subset \mathbb{P}^{d}$ be the dense torus of $\mathbb{P}^d$, i.e., $T \simeq (\mathbb{C}^*)^d \subset \mathbb{P}^d$. In this case, we have
\[ \chi(V(F)) \, = \, \chi(\mathbb{P}^d) - \chi(T) + \chi(V(F) \cap T) \, = \,  (d+1) + \chi(V(f) \cap T).\]
Hence, an Euler stratification for the projective hypersurface $V(F)$ is an Euler stratification for the very affine hypersurface $V(f) \cap T$. Very affine varieties are central in tropical geometry \cite{maclagan2021introduction}. They appear as statistical models in algebraic statistics \cite{clarke2023matroid,huh2014likelihood} and as integration spaces in particle physics \cite{fourlectures}. In these applications, the Euler characteristic of the very affine variety at hand plays a crucial role; see for instance \cite[Theorem 1.7]{huh2014likelihood} and \cite[Theorem 3.14]{fourlectures}. In statistics, the absolute value of the Euler characteristic coincides with the maximum likelihood degree of the corresponding model, which measures the algebraic complexity of maximum likelihood estimation \cite{catanese2006maximum}. In physics, $|\chi(V(f))\cap T|$ is the dimension of a vector space of integrals, and it measures the complexity of integration by parts reduction~\cite{agostini2022vector,bitoun2019feynman,fourlectures}.

In the situation of Equation \eqref{eq:factorscase}, we may consider equisingular loci for the zero set of $f(x,z)$ in $T$. If, in addition, the parameters $z$ appear as coefficients of $f$ (as in \eqref{eq:Fconics}), then this leads us to study $A$-discriminants and principal $A$-determinants as defined by Gel'fand, Kapranov and Zelevinsky (GKZ) \cite[Chapters 9--10]{gelfand2008discriminants}. We illustrate this for our conics example. 

\begin{example} \label{ex:conicsveryaff}
    Let $f(x,z)$ be the polynomial in \eqref{eq:Fconics}. Its exponents are the columns of 
    \[ A \, = \, \begin{pmatrix}
        2 & 1 & 1 & 0 & 0 & 0 \\ 0 & 1 & 0 & 2 & 1 & 0\\ 0 & 0 & 1 & 0 & 1 & 2
    \end{pmatrix}.\]
    The cubic defining equation $\Delta$ of the hypersurface $\nabla$ in Example \ref{ex:conics_intro} is the $A$-discriminant associated to this matrix. By \cite[Theorem 13]{amendola2019maximum}, the Euler characteristic of $V(f(x,z)) \cap T$ equals $-4$, unless the principal $A$-determinant vanishes:
    \begin{equation} \label{eq:EAconics} E_A \, = \, \Delta \cdot (z_1^2-4z_0z_3)\cdot (z_2^2-4z_0z_5) \cdot (z_4^2-4z_3z_5) \cdot z_0 \cdot z_3 \cdot z_5 \, = \, 0.\end{equation}
    For a generic point in the $4$-dimensional variety $\nabla_\chi  = \{ z \in \mathbb{P}^5 \, : \, E_A(z) = 0\}$, the Euler characteristic is $\chi(V(f(x,z)) \cap T) = -3$. Overall, the possible Euler characteristics are $0, -1, -2, -3, -4$. In Section \ref{sec:5}, we will decompose $\nabla_\chi$ into 70 Euler strata. 
\end{example}

Esterov studies Euler characteristics and multisingularity strata in the GKZ setting in \cite{esterov2013discriminant,esterov2017characteristic}. In particular, in \cite[Section 1.2]{esterov2013discriminant} he coins the term ``Euler discriminant'', which was picked up in the context of Feynman integrals in \cite{fevola2024principal}. In this paper, the Euler discriminant variety $\nabla_\chi$ is the subvariety of $Z$ obtained as the closure of all $z$ for which the Euler characteristic $\chi(V(F(x,z)))$ takes non-generic values. If $\nabla_\chi$ is defined by a single equation $\Delta_\chi = 0$, then we call $\Delta_\chi$ the Euler discriminant polynomial. Since computing an Euler stratification is an iterated computation of Euler discriminants, understanding these discriminants is crucial to our story. The formulae for $\chi(V(F))$ stated by Dimca and Papadima \cite{dimca2003hypersurface} and Huh \cite{huh2012milnor,huh2013maximum} are a key to success for computing Euler discriminants using computer algebra. 

The Euler stratifications considered here are coarsenings of the more classical Whitney stratification of (the parameter space of) a family of varieties \cite{whitney1964tangents}. More precisely, by Thom's first isotopy lemma \cite[Proposition 11.1]{mather2012notes}, the strata of a Whitney stratification parametrize varieties of constant topological type, and thus in particular of constant Euler characteristic. Brown's definition of the Landau variety in \cite[Definition 54]{brown2009periods} is based on Thom's isotopy lemma; it is the union of all Whitney strata of codimension one. This variety captures the singular locus of Feynman integrals, viewed as multi-valued functions of kinematic parameters. The Landau variety contains the Euler discriminant, but the two do not always coincide, see Examples \ref{ex:counterwhitney1} and \ref{ex:counterwhitney2}. However, in physics, they often do.

Recent efforts by Helmer and Nanda have lead to symbolic algorithms for computing Whitney stratifications \cite{helmer2023conormal}. This was applied to Feynman integrals and Landau varieties in \cite{helmer2024landau}. That paper shows that the algorithms work well in small examples, but they run out of steam for more challenging integrals, such as those tackled in \cite{fevola2024principal}. However, the efficient symbolic-numerical methods in \cite{fevola2024principal} often lead to incomplete results, resulting in a variety that is strictly contained in the Euler discriminant, called the Principal Landau determinant. 

\vspace{-0.2cm}
\paragraph{Contributions and outline.} Section \ref{sec:2} gives a definition of Euler stratifications (Definition \ref{def:ES}), proves existence for a general class of hypersurface families (Lemma \ref{prop:Xproj}) and presents some first examples. We discuss the relation to Whitney stratifications and multisingularity strata.
Section \ref{sec:3} revisits coincident root loci for binary forms. We summarize known results and deduce an explicit characterization of the Euler stratification (Theorems \ref{thm:StratiRootsProj} and \ref{thm:strataXveryaffP1}). We also present a new algorithm for the Euler stratification of hyperplane sections of smooth projective curves (Algorithm \ref{alg:stratifyPointsCurve}). In Section \ref{sec:4}, we develop general algorithms (Algorithms \ref{alg:stratify_proj}, \ref{alg:polardisc} and \ref{alg:eulerdiscproj}) for computing the Euler stratification of projective hypersurface families. For families of plane curves with isolated singularities, we prove that the Euler discriminant equals the polar discriminant (Proposition \ref{prop:PolarEqEuler}). Section \ref{sec:5} is dedicated to the very affine setting. We prove that the degree of the Gauss map counts the critical points of a logarithmic potential (Proposition \ref{prop:ncrit}). We relate this to June Huh's result (Theorem \ref{thm:huh2}) using likelihood degenerations \cite{agostini2023likelihood}. We prove two structural results in Theorem \ref{thm:Zsclosed} and Proposition \ref{prop:strictcontained}, and present the specialized Algorithm \ref{alg:eulerdiscveryaff} for computing Euler discriminants in this setting. In Section \ref{sec:6}, we present a gallery of examples, including coincident root loci of binary octics (Section \ref{sec:61}), matroid stratifications of bilinear forms (Section \ref{sec:62}), Landau varieties of Feynman integrals (Section \ref{sec:63}), maximum likelihood stratifications for toric Fano varieties (Section \ref{sec:64}) and Hirzebruch surfaces (Section \ref{sec:65}).  Our code and computational results are available online at MathRepo:
\begin{equation} \label{eq:mathrepo} \tag{$\star$}
\text{\url{https://mathrepo.mis.mpg.de/EulerStratifications}}
\end{equation}

\section{Definitions and first examples} \label{sec:2}

Our first task is to define Euler stratifications, and to show that they exist for sufficiently general families of varieties. A \emph{quasi-projective variety} is a subset $X \subset \mathbb{P}^d$ of projective space which can be written as $X = V \setminus W$, where $V,W \subset \mathbb{P}^d$ are closed subvarieties. 
\begin{definition} \label{def:ES}
    Let $\mathcal{X}, Z$ be quasi-projective varieties, with $Z$ irreducible. Let $\pi: \mathcal{X} \rightarrow Z$ be a surjective morphism. An \emph{Euler stratification} of $\pi$ is a finite set $\mathscr{S}$ of quasi-projective subvarieties of $Z$ satisfying the following three conditions:
\begin{enumerate}
\setlength\itemsep{0em}
    \item when $S \neq S' \in \mathscr{S}$, then $S \cap S' = \emptyset$, and $\bigsqcup_{S \in \mathscr{S}} S = Z$,
    \item for each stratum $S \in \mathscr{S}$, the closure $\overline{S}$ is a union of strata: $\overline{S} = \bigsqcup_{S' \subseteq \overline{S}} S'$, 
    \item the Euler characteristic of the fiber $\chi(\pi^{-1}(z))$ is constant for $z \in S$. 
\end{enumerate}
The set $\mathscr{S}$ is partially ordered as follows: $S \preceq S'$ if and only if $\overline{S}\subseteq \overline{S'}$.
\end{definition}

The quasi-projective subvarieties $S$ in an Euler stratification $\mathscr{S}$ are called \emph{(Euler) strata}. Their closures $\overline{S}$ are the \emph{closed (Euler) strata}. Euler strata need not be irreducible or smooth. 

We make the following concrete choices for $\mathcal{X}$ and $Z$. Let $F \in \mathbb{C}[z][x]$ be a bi-homogeneous polynomial in the variables $x = (x_0, \ldots, x_d)$ and parameters $z = (z_0, \ldots, z_m)$. Let $Z \subseteq \mathbb{P}^m$ be any irreducible quasi-projective subvariety. We define
\begin{equation} \label{eq:Xproj} \mathcal{X}_F \, = \, \{ (x, z) \in \mathbb{P}^d \times Z \, : \, F(x,z) \neq  0 \}. \end{equation}
We want to stratify the family of hypersurface complements given by the projection \begin{equation} \label{eq:piproj} \pi_F: \mathcal{X}_F \rightarrow Z, \quad (x, z) \longmapsto z.
\end{equation}
We denote the fibers of this map by $\X_{F,z} = \pi_F^{-1}(z)$. We assume that the set $\{z \in Z \, : \, F(x,z) \equiv 0 \}$ is empty, so that all fibers are $d$-dimensional. 
\begin{remark} \label{rem:closedfibers}
    We may alternatively consider the family of hypersurfaces \[ \mathcal{X}_F^c = \{(x,z) \in \mathbb{P}^d \times Z \, : \, F(x,z) = 0 \} \]
    with closed fibers $\X_{F,z}^c = \mathbb{P}^d \setminus \X_{F,z}$. The equality $\chi(\X_{F,z}) = \chi(\mathbb{P}^d) - \chi(\X_{F,z}^c) = d + 1 - \chi(\X_{F,z}^c)$ implies that an Euler stratification of $\pi_F$ is one of $\pi_F^c : \mathcal{X}_F^c \rightarrow Z$ and vice versa. 
\end{remark}

\begin{example}[Conics in $\mathbb{P}^2$] \label{ex:conicsinP2}
Consider the family of conics $\mathcal{X}_F = \{ (x,z) \in \mathbb{P}^2 \times \mathbb{P}^5 \, : \, F(x,z) \neq  0 \}$, where $F$ is as in \eqref{eq:Fconics}, with its coordinate projection $\pi: \mathcal{X}_F \rightarrow \mathbb{P}^5$. By the discussion in Example \ref{ex:conics_intro}, the poset $\mathscr{S}$ consists of three strata, of dimensions five, four, and two. The closed four-dimensional stratum is given by the discriminant $\Delta = 0$, with $\Delta$ as in \eqref{eq:conicdiscriminant}. The smallest stratum is, up to scaling coordinates, the second Veronese embedding of $\mathbb{P}^2$, whose binomial ideal is given by the $2 \times 2$-minors of $M(z)$ from \eqref{eq:Mz}. Notice that there is no stratum of dimension three, and the Euler characteristic drops by one on the discriminant.
\end{example}

Example \ref{ex:conicsinP2} illustrates that the loci $Z_k = \{ z \in Z\, : \, \chi(\X_{F,z}) = k \}$ of constant Euler characteristic are not necessarily quasi-projective, even when $\mathcal{X}_F$ and $Z$ are quasi-projective. We will show that they are always \emph{constructible}. A constructible set is a finite union of quasi-projective varieties $(V_1 \setminus W_1) \cup \cdots \cup (V_\ell \setminus W_\ell)$. 

\begin{proposition} \label{prop:Xproj}
    Let $F,\mathcal{X}_F, Z$ be as in \eqref{eq:Xproj} and let $\pi_F : \mathcal{X}_F \rightarrow Z$ be the coordinate projection \eqref{eq:piproj} with fibers $\X_{F,z} = \pi_F^{-1}(z)$. For any integer $k$, the set $Z_k = \{ z \in Z \, :\, \chi(\X_{F,z}) = k \}$ is constructible. In particular, there exists an Euler stratification of the map $\pi_F$.
\end{proposition}
We prove Proposition \ref{prop:Xproj} using a result of Dimca and Papadima \cite[Theorem 1]{dimca2003hypersurface}.

\begin{theorem} \label{thm:Dimca-Papadima}
    Let $F \in \mathbb{C}[x_0, \ldots, x_d]$ be any non-constant homogeneous polynomial, and let $D(F) = \{ x \in \mathbb{P}^d \, : \, F(x) \neq 0 \}$ be its hypersurface complement. Consider the Gauss~map 
    \[ \nabla F\, : D(F) \longrightarrow \mathbb{P}^d, \quad x \longmapsto \left( \frac{\partial F}{\partial x_0}(x):\cdots: \frac{\partial F}{\partial x_d}(x) \right ).\] 
    We have $\deg \nabla F = (-1)^d \cdot \chi(D(F) \setminus H)$, where $H \subset \mathbb{P}^d$ is a generic hyperplane.
\end{theorem}

\begin{proof}[Proof of Proposition \ref{prop:Xproj}]
    We may assume that ${\cal X}_{F,z}^c \neq \emptyset$ for each $z \in Z$. Indeed, the set $W = \{ z \in Z \, : \, {\cal X}_{F,z}^c = \emptyset \}$ is closed, and we may replace $Z$ by $Z \setminus W$. 
    
    For any $z_0\in Z$, there exists a hyperplane $H\subset \mathbb{P}^d$ such that $H$ intersects ${\cal X}_{F,z_0}^c$ transversely. The same hyperplane intersects ${\cal X}_{F,z}^c$ transversely for $z$ in a dense open set $U\subseteq Z$. Therefore, it suffices to show that $Z_k \cap U$ is constructible. 
    The formula for $\chi(D(F) \setminus H)$ from Theorem \ref{thm:Dimca-Papadima} implies that
    \begin{align*}
        (-1)^d \cdot \deg \nabla F(x,z) \, &= \, \chi({\cal X}_{F,z} \setminus H) \\
        & = \, \chi({\cal X}_{F,z})- \chi({\cal X}_{F,z} \cap H)\\
        & = \, \chi(\X_{F,z}) - (\chi(\mathbb{P}^{d-1}) - \chi(\X_{F,z}^c \cap H))
    \end{align*}
    for all $z \in U$ (notice that ${\cal X}_{F,z} = D(F(x,z))$ for any $z \in Z$). Using $\chi(\mathbb{P}^{d-1}) = d$ we find
    \begin{equation} \label{eq:formulaDP}
        \chi(\X_{F,z}) \, = \, d - \chi(\X_{F,z}^c \cap H) + (-1)^d \cdot \deg \nabla F(x,z).
    \end{equation}
    For any $k$, we claim that the set $W_k = \{ z \in Z\, : \, \deg \nabla F(x,z) = k \}$ is constructible. To see this, note that the ramification locus of $\{ (x,z,b) \in {\cal X}_F \times \mathbb{P}^d_b \, :\, \nabla F(x,z) = b \} \rightarrow Z \times \mathbb{P}^d_b$ is closed, and its image $B \subset Z \times \mathbb{P}^d_b$ is constructible by Chevalley's theorem. 
    Let $B' \subset Z$ be the constructible subset consisting of the points $z \in Z$ whose fiber along the projection $B \rightarrow Z$ is $d$-dimensional. By construction, we have $Z \setminus B' = W_{n^*}$, where $n^*$ is the generic degree of $\nabla F(x,z)$ for $z \in Z$. The same argument applies when we replace $Z$ by any of the irreducible components of $\overline{B'} \subset Z$. The claim follows by iterating this process.
    
    When $d = 1$, we have $\X_{F,z}^c \cap H = \emptyset$ and $Z_k \cap U$ is constructible by \eqref{eq:formulaDP}. We proceed by induction on $d$: if the statement is true in dimension $d-1$, then $\{ z \in U \, : \, \chi(\X_{F,z}^c \cap H) = k \}$ is constructible for any $k$ (Remark \ref{rem:closedfibers}), which implies by \eqref{eq:formulaDP} that $Z_k \cap U$ is~constructible. 
\end{proof}

As mentioned in the Introduction, the applications we have in mind require Euler stratifications of families of \emph{very affine} hypersurface complements. That is, we are interested in the Euler characteristic of the complement of the zero locus of a polynomial $f$ in the torus 
\[ T \, = \,  D(x_0 \cdots x_d) \, = \,  \{(x_0: \cdots: x_d) \in \mathbb{P}^d \, : \, x_0 \cdots x_d \neq 0 \} \, \simeq \,  (\mathbb{C}^*)^d . \] 
We continue to assume that $Z \subseteq \mathbb{P}^m$ is an irreducible quasi-projective variety, and we set 
\begin{equation} \label{eq:Xveryaff} \mathcal{X}^*_f \, = \, \{ (x, z) \in T \times Z \, : \, f(x,z) \neq 0 \} \, = \, {\cal X}_F,
 \end{equation} 
where $F = x_0 \cdots x_d \, f$. 
The following is a consequence of Proposition \ref{prop:Xproj} and ${\cal X}^*_f = {\cal X}_F$. 
\begin{lemma} \label{lem:Xveryaff}
    Let $f \in \mathbb{C}[z][x]$ be any bi-homogeneous polynomial and let  $Z, \mathcal{X}_f^*$ be as in \eqref{eq:Xveryaff}. The coordinate projection $\pi_f: \mathcal{X}_f^* \rightarrow Z$ admits an Euler stratification. 
\end{lemma}

\paragraph{Whitney stratifications.}
    Proposition \ref{prop:Xproj} can also be proved via the theory of \emph{Whitney stratifications}, originally developed by Thom \cite{thom1964local} and Whitney \cite{whitney1964tangents}. This proof strategy works more generally when $\pi:{\cal X} \rightarrow Z$ is a proper morphism of (abstract) varieties. A statement that fits our scope precisely is hard to find in the literature, and a self-contained explanation would be too much of a digression.  Since the identity \eqref{eq:formulaDP} is useful in later sections, we chose to include the proof above and limit ourselves to a sketch of the general argument. 

    In \cite[Theorem 19.2]{whitney1964tangents}, Whitney states that every \emph{variety} admits what is now called a Whitney stratification. Here \emph{variety} means complex analytic variety. The fact that every constructible set admits a Whitney stratification whose strata are again constructible is stated explicitly and with an outline of proof in \cite[pp.\ 336--337]{wall2006regular}. A Whitney stratification of a proper morphism $\pi: {\cal X} \rightarrow Z$ (such as our morphism $\pi^c_F: {\cal X}_F^c \rightarrow Z$) is a Whitney stratification of ${\cal X}$ and $Z$ such that $\pi$ maps each stratum $S_{\cal X}$ of ${\cal X}$ into a stratum $S_Z$ of $Z$, and the restriction $\pi_{S_{\cal X}}: S_{\cal X} \rightarrow S_Z$ is a submersion \cite[Definition 6.1]{helmer2023conormal}. By Thom's first isotopy lemma \cite[Proposition 11.1]{mather2012notes}, the topology of the fibers of $\pi$ is constant on each stratum $S_Z$. In particular, the Euler characteristic is constant. It follows that the Whitney stratification of $\pi$ is an Euler stratification. The converse is not true: a priori, Euler strata are unions of Whitney strata. We provide an example in the projective and very affine case.

\begin{example} \label{ex:counterwhitney1}
Consider the following family of cuspidal cubic plane curves:
        \[\X^c_F = \left\{ (x,z) \in \P^2\times \P^1 \midcolon z_0x_0^3 + z_1x_1^2x_2 = 0 \right\}.\]
For $z\notin \{(1:0),(0:1)\}$, the fiber $(\pi^c_F)^{-1}(z)$ has a cusp singularity at the origin, implying $\chi\left( (\pi^c_F)^{-1}(z) \right) = 2$ (see Example \ref{ex:Chi(cusp)} below). For $z=(1:0)$, the cubic degenerates into a triple line, also with Euler characteristic 2. Thus, an Euler stratification for $\pi^c_F$ is given by $\{(0:1)\} \sqcup (\P^1 \setminus \{(0:1)\})$, whereas $\{(1:0)\}$ necessarily forms a separate Whitney stratum.
\end{example}

\begin{example}\label{ex:counterwhitney2}
Consider the family of very affine curve complements
        \[(\X^*_f)^c = \left\{ (x,z) \in T \times \P^1 \midcolon (x_1 + x_2)(z_0x_1 + z_1x_2) = 0 \right\}.\]
        Here, $T$ is the torus $(\mathbb{C}^*)^2 \simeq T \subset \mathbb{P}^2$. For $z\notin \{(1:0),(1:1),(0:1)\}$, the fiber of $({\cal X}_f^*)^c \rightarrow \mathbb{P}^1$ consists of two lines intersecting at the origin. For $z\in \{(1:0),(1:1),(0:1)\}$, the fiber consists of only one line through the origin. The Euler characteristic is always zero, hence there is only a single Euler stratum. However, the topological types are different.
\end{example}

\paragraph{Multisingularity strata.} \label{par:multisingularity}
As seen in previous examples, Euler stratifications are closely related to the singularity structure of $\X$. In the following we recall some basic notions of singularity theory following \cite{greuel2007introduction} and elaborate on the relation to Euler stratifications.\par 
Let $U\subseteq \C^d$ be an analytic open subset, $f\colon U\rightarrow \C$ a holomorphic function and $X = \V(f) \subset U$ the hypersurface defined by $f$. The set of singular points of $X$ is given by \[ \sing(X) \, = \,  \left\{ x\in U \midcolon f(x) = \frac{\partial f}{\partial x_1}(x) = \dots = \frac{\partial f}{\partial x_d}(x) = 0 \right\}. \] A point $x\in \sing(X)$ is called an isolated singularity if there exists a neighborhood $V$ of $x$ such that $(\sing(X) \cap V) \setminus \{x\} = \emptyset$. In singularity theory, isolated singularities are typically studied up to right equivalence (analytic change of coordinates) or contact equivalence (isomorphism of factor algebras). Let $j(f)$ denote the ideal sheaf $ \langle \frac{\partial f}{\partial x_1}, \dots, \frac{\partial f}{\partial x_d} \rangle \cdot \O(U)$ and define the algebra $M_{f,x} := \O_{\C^d,x} / j(f)\O_{\C^d,x}$, called the \emph{Milnor algebra} of $f$ at $x$. Its dimension as a $\C$-vector space is the \emph{Milnor number} of $f$ at $x$. We write $\mu(f,x) = \dim_{\C} M_{f,x}$.   This number is a topological invariant of the singularity and plays a crucial role in our considerations. Note that $\mu(f,x) > 0$ if and only if $x$ is a singular point of $f$ and that $\mu(f,x)$ is finite by an application of the Hilbert--R\"uckert Nullstellensatz. If $X$ has only isolated singularities $y_1,\dots,y_m$, then we define the \emph{total Milnor number} to be $\mu(X) := \sum_{i=1}^m \mu(f,y_i)$. The connection to Euler stratifications is established by the following statement \cite[Corollary~1.7]{parusinski1988generalization}.

\begin{proposition}
    Let $M$ be a $d$-dimensional smooth complex projective variety and let $\mathcal{L}$ be a line bundle on $M$. For two sections $s_1,s_2\in H^0(M,\mathcal{L})$, let $X_1=V(s_1)$ and $X_2=V(s_2)$ denote their respective zero loci. If $X_1$ and $X_2$ have only isolated singularities, then we have the equality $\mu(X_1) - \mu(X_2) = (-1)^d (\chi(X_1) - \chi(X_2))$.
\end{proposition}

\begin{example}
    \label{ex:Chi(cusp)}
    A smooth cubic plane curve has Euler characteristic zero. A cusp singularity has Milnor number two. Therefore, a cuspidal plane cubic has Euler characteristic two. The full stratification of plane cubics is shown in Figure \ref{fig:cubics}. Each stratum of ${\cal X}_F^c \rightarrow Z = \mathbb{P}^9 \simeq \mathbb{P}(\mathbb{C}[x_0,x_1,x_2]_3)$ corresponds to a circle in the figure. Strata with isolated singularities are labeled by their singularity type. Below each circle, we record the Euler characteristic of ${\cal X}_{F,z}^c$ on the stratum. The edges between strata indicate the poset relation. The Euler characteristic of a fiber ${\cal X}_{F,z}^c$ is invariant under the action of ${\rm GL}_{3}(\mathbb{C})$ on ternary cubics. The ring of invariants is generated by $I_4$ and $I_6$, which are homogeneous polynomials of degree 4, resp.~6, in the 10 coefficients $z_0, \ldots, z_{9}$. These generators are unique up to scaling, and $I_4$ is called the \emph{Aronhold invariant} \cite[Example 11.12]{michalek2021invitation}. The closed stratum of nodal cubics ($A_1$) is defined by the discriminant $I_4^3-I_6^2$, and the cuspidal cubics ($A_2$) are defined by $I_4 = I_6 = 0$. 
    \begin{figure}
        \centering
        \includegraphics[width=0.90\linewidth]{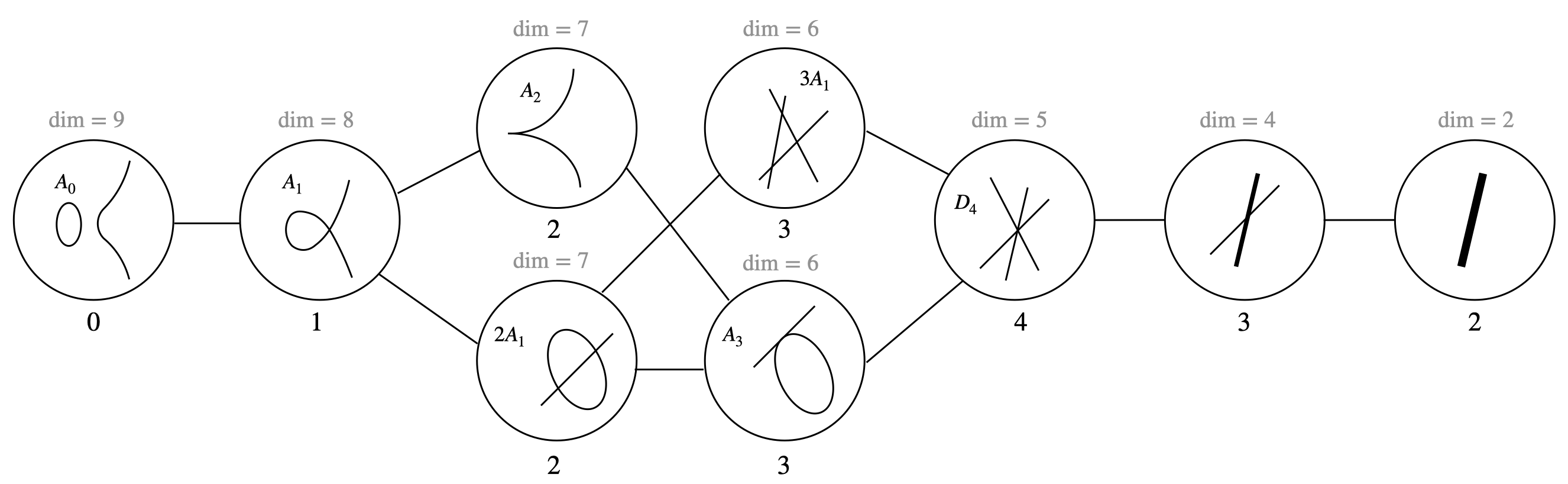}
        \caption{The Euler stratification of plane cubics.}
        \label{fig:cubics}
    \end{figure}
\end{example}

This shows that for families with only isolated singularities, any Euler stratum is a union of \emph{multisingularity strata}, i.e.,\ loci where the general members of the hypersurface family have a prescribed topological type of singularities. However, multisingularity strata are, in general, extremely hard to compute. Even for the case of plane curves, a complete description of the multisingularity strata are out of reach \cite{diaz1988geometry}. In \cite{esterov2017characteristic}, Esterov provides formulae for the tropicalization of the strata for two $A_1$ singularities and one $A_2$ singularity.

\section{Points on lines and curves} \label{sec:3}

The first part of this section studies the case $d = 1$, corresponding to points on the projective line. At the end of the section, we replace $\P^1$ by a smooth projective curve.
For points on $\P^1$, it is possible to enumerate and parametrize all Euler strata. However, finding implicit equations is challenging.   
Let $n$ be a positive integer. Simplifying notation slightly,~we~set 
\begin{equation} \label{eq:XprojP1}
    \X^c = \left\{ (x,z) \in \P^1\times \P^n \midcolon z_0x_0^n + z_1x_0^{n-1}x_1 + \dots + z_nx_1^n = 0\right\}.
\end{equation}
We study the family given by the projection $\pi\colon\X^c\rightarrow Z$. The topological Euler characteristic of a fiber $\X^c_z = \pi^{-1}(z)$ is equal to the number of points of $\X^c_z$, ignoring multiplicities. Hence, Euler strata are loci where certain roots of $F = z_0x_0^n + \cdots + z_n x_1^n$ coincide. We label these strata by integer partitions of $n$ as follows. Let $\lambda = (\lambda_1,\dots,\lambda_k)$ be such a partition. For any $k > 0$, the set $Z_k = \left\{ z\in Z\midcolon \chi(\X^c_z) = k \right\}$ is a disjoint union of strata $Z_k = \coprod_{|\lambda|=k} S_{\lambda}$ ranging over all partitions of $n$ with length $k$. Here, $S_{\lambda}$ denotes the set
\begin{equation} \label{eq:Slambda}
    S_{\lambda} = \left\{z\in \P^n\colon {\X}^c_{z} \text{ has } k \text{ distinct points with multiplicities } \lambda_1,\dots,\lambda_k\right\}.
\end{equation}
The sets $S_{\lambda}$ are called \emph{coincident root loci}, \emph{multiple root loci} or \emph{Brill--Gordan loci} in the literature, see e.g.\ \cite{chipalkatti2003equations, feher2006coincident, kurmann2012some, lee2016duality, weyman1989equations}. In what follows we review the most important results. \par 

Let $\prec$ denote the partial order on the set of partitions given by refinement. That is, $\lambda = (\lambda_1,\dots,\lambda_k)$ \emph{refines} $ \mu =(\mu_1,\dots,\mu_l)$, written $\lambda \prec \mu$, if and only if there exists a partition $(I_1,\dots,I_l)$ of $[k]$ such that for any $1\leq i\leq l:~ \mu_i =\sum_{j\in I_i}\lambda_j$. An example of the resulting lattice for partitions of five can be seen in Figure \ref{fig:young_diagrams}. The partitions are labeled by Young diagrams, and by strings like $21^3$, representing $\lambda = (2,1,1,1)$. Each circle represents an irreducible Euler stratum of $\pi\colon \X^c \subset \P^1 \times \P^5 \rightarrow \P^5$. 
\begin{figure}
    \centering
    \includegraphics[width = 11cm]{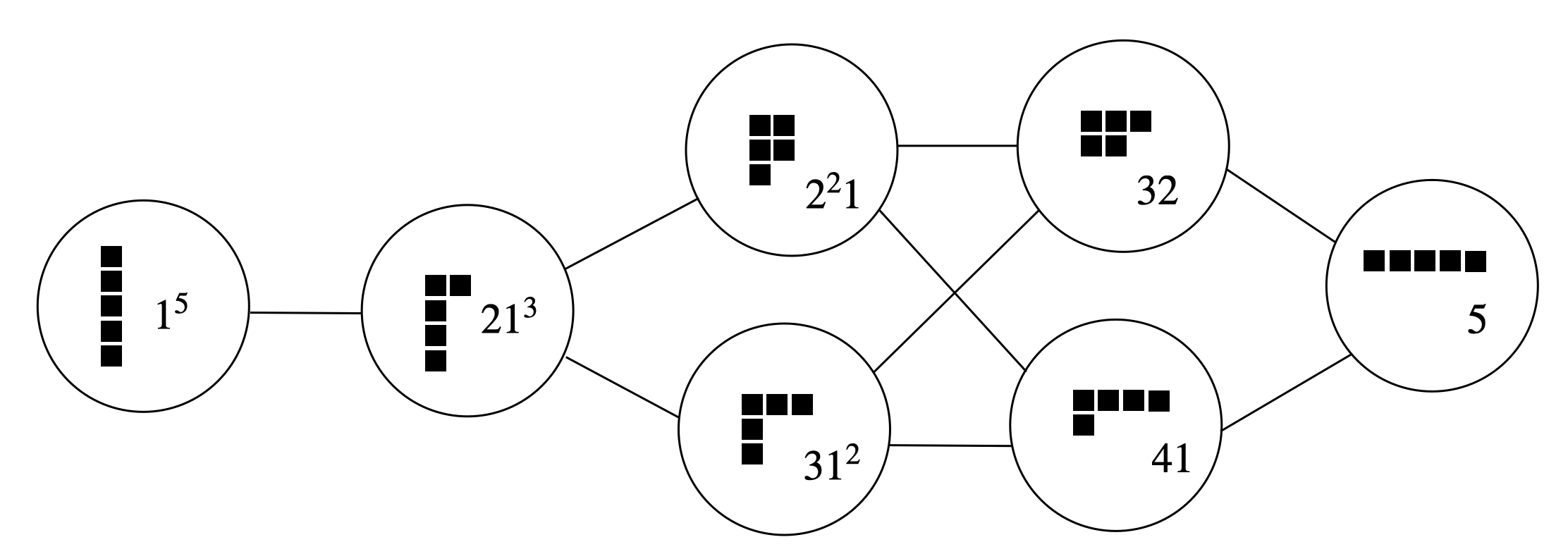}
    \caption{Young diagrams with five boxes index the strata of $\pi\colon \X^c \subset \P^1\times\P^5\rightarrow \P^5$. }
    \label{fig:young_diagrams}
\end{figure}
The following theorem is well known. 

\begin{theorem}
\label{thm:StratiRootsProj}
Let $\X^c$ be as in \eqref{eq:XprojP1} and consider the projection  
$\pi\colon \X^c\rightarrow Z=\P^n$. 
\begin{enumerate}
    \setlength\itemsep{0em}
    \item[(i)] The set of coincident root loci $\left\{ S_{\lambda} \right\}_{\lambda}$ ranging over all partitions $\lambda$ of $n$, ordered by refinement so that $S_{\lambda} \prec S_{\lambda'} \Leftrightarrow \lambda' \prec \lambda$, forms an Euler stratification of $\pi$.
    \item[(ii)] The Zariski closure $\nabla_{\lambda}$ of $S_{\lambda}$ is an irreducible projective variety. It equals the disjoint union $\nabla_{\lambda} = \coprod_{\lambda \prec \lambda^{\prime}} S_{\lambda^{\prime}}$, where $\prec$ denotes refinement of partitions.
    \item[(iii)] For $z\in S_{\lambda}$, the Euler characteristic of the fiber is $\chi(\X_z^c) = |\lambda|$. This equals the dimension $\dim S_{\lambda} = \dim \nabla_\lambda = |\lambda|$ and the height of the corresponding Young diagram.
\end{enumerate}
\end{theorem}

Hilbert proved in \cite{hilbertDegreeUniv} that the degree of the $k$-dimensional projective variety $\nabla_{\lambda}$ is
\[
    \deg(\nabla_{\lambda}) = \frac{k!}{m_1!m_2!\dots m_n!}\cdot \lambda_1\lambda_2\dots \lambda_k,
\]
where $m_j := \#\{i\colon \lambda_i = j\}$. In general, not much is known about the structure of the ideal $I(\nabla_{\lambda})$ defining $\nabla_{\lambda}$. For the case where $\lambda$ is of the form $\lambda = (1^{n-a},a)$ for $a\geq \left\lfloor n/2\right\rfloor +2$, Weyman showed that $I(\nabla_{\lambda})$ is generated in degree $\leq 4$ \cite{weyman1989equations}. Moreover, for $|\lambda|=2$, Abdesse-lam and Chipalkatti proved that $I(\nabla_{\lambda})$ is generated in degree $\leq 4$ \cite[Proposition 20]{abdesselam2006bipartite}.\par 
Chipalkatti describes a generating set for $I(\nabla_{\lambda})$ in terms of covariant forms \cite{chipalkatti2004invariant}. We can deduce an upper bound on the number of generators of $I(\nabla_{\lambda})$ using \cite[Theorem 3.5 and Proposition 3.6]{chipalkatti2004invariant} as follows: Let $\mathcal{C}_{\lambda}$ be the set of all $n$-partitions $\mu \nprec \lambda$ that are minimal with this property, i.e.,\ $\mathcal{C}_{\lambda} = \{\mu \nprec \lambda \colon \nu \prec \mu \Rightarrow \nu \prec \lambda\}$. Then $I(\nabla_{\lambda})$ can be generated by $n!\cdot |\mathcal{C}_{\lambda}|$ many polynomials. For example, consider the partition $\lambda = (3,2)$. Then $\mathcal{C}_{\lambda} = \{(4,1)\}$ and hence $I(\nabla_{\lambda})$ can be generated by $\leq 5!$ polynomials. This bound is not optimal, as $I(\nabla_{\lambda})$ is minimally generated by 28 polynomials, see \cite[Table 1]{lee2016duality}.\par 

Computing defining equations for $\nabla_{\lambda}$ quickly becomes a challenging task. For example, $I(\nabla_{2 1^3})$ (appearing in Figure \ref{fig:young_diagrams}) is the zero locus of the discriminant of the binary quintic, which is a degree eight polynomial with 59 terms:
\[
    z_{1}^{2}z_{2}^{2}z_{3}^{2}z_{4}^{2}-4\,z_{0}z_{2}^{3}z_{3}^{2}z_{4}^{2}-4\,z_{1}^{3}z_{3}^{3}z_{4}^{2}+18\,z_{0}z_{1}z_{2}z_{3}^{3}z_{4}^{2}-27\,z_{0}^{2}z_{3}^{4}z_{4}^{2}-4\,z_{1}^{2}z_{2}^{3}z_{4}^{3}+ \cdots +3125\,z_{0}^{4}z_{5}^{4}.
\]
In \cite{lee2016duality}, the authors present a table with the degrees of the minimal generators for the ideals $I(\nabla_{\lambda})$ for $n \leq 7$. This can be done efficiently (and heuristically) by using finite field computations. In Section \ref{sec:61}, we extend these results to $n = 8$. Moreover, we compute all generators for $n = 1,\ldots,8$ over $\mathbb{Q}$.  
For this, we use the fact that $S_{\lambda}$ is parametrized~by
\begin{equation} \label{eq:paramSlambda}
    (\P^1)^k \dashrightarrow S_{\lambda},\quad ((a_1:b_1),\ldots, (a_k: b_k)) \, \longmapsto  \, (a_1x_0 + b_1x_1)^{\lambda_1}\cdots(a_kx_0 + b_kx_1)^{\lambda_k}.
\end{equation}

As indicated in the Introduction, applications in physics and statistics require studying Euler stratifications of families of very affine hypersurfaces in the algebraic torus. We modify the previous constructions by disregarding points at zero and infinity. That is, we consider
\begin{equation} \label{eq:XveryaffP1}
    (\X^*)^c = \{(t,z) \in \mathbb{C}^* \times \mathbb{P}^n \midcolon z_0 + z_1t + \cdots + z_n t^n = 0 \},
\end{equation}
with projection $\pi\colon  (\X^*)^c \rightarrow Z=\P^n$. Notice that, while previously all Euler strata were necessarily stable under the action of $\mathrm{GL}(2)$ on binary forms of degree $n$, this symmetry is now broken. This leads to a significant increase in the number of strata. \par 

Strata are now labeled by two integers $m_0, m_\infty$ and a partition of $n - m_0 - m_\infty$. These integers satisfy $m_0, m_\infty \geq 0$ and $m_0 + m_\infty \leq n$. For instance, the stratum corresponding to $m_0 = 0, m_\infty = 2$ and the partition $21$ of $3$ consists of binary quintics with a root of multiplicity two at infinity, and two roots in $\mathbb{C}^*$, one of which has multiplicity two and one of which is simple. In Figure \ref{fig:strataC1}, we label such a stratum for $n = 5$ by the corresponding partition, with the integers $m_0$ and $m_\infty$ written to the left, resp.~to the right of it. The above stratum is $0|21|2$. The poset from Figure \ref{fig:young_diagrams} appears on the right side of Figure \ref{fig:strataC1}, considering only strata with $m_0 = m_\infty = 0$. 
\begin{figure}[!ht]
    \centering
    \includegraphics[width = 16cm]{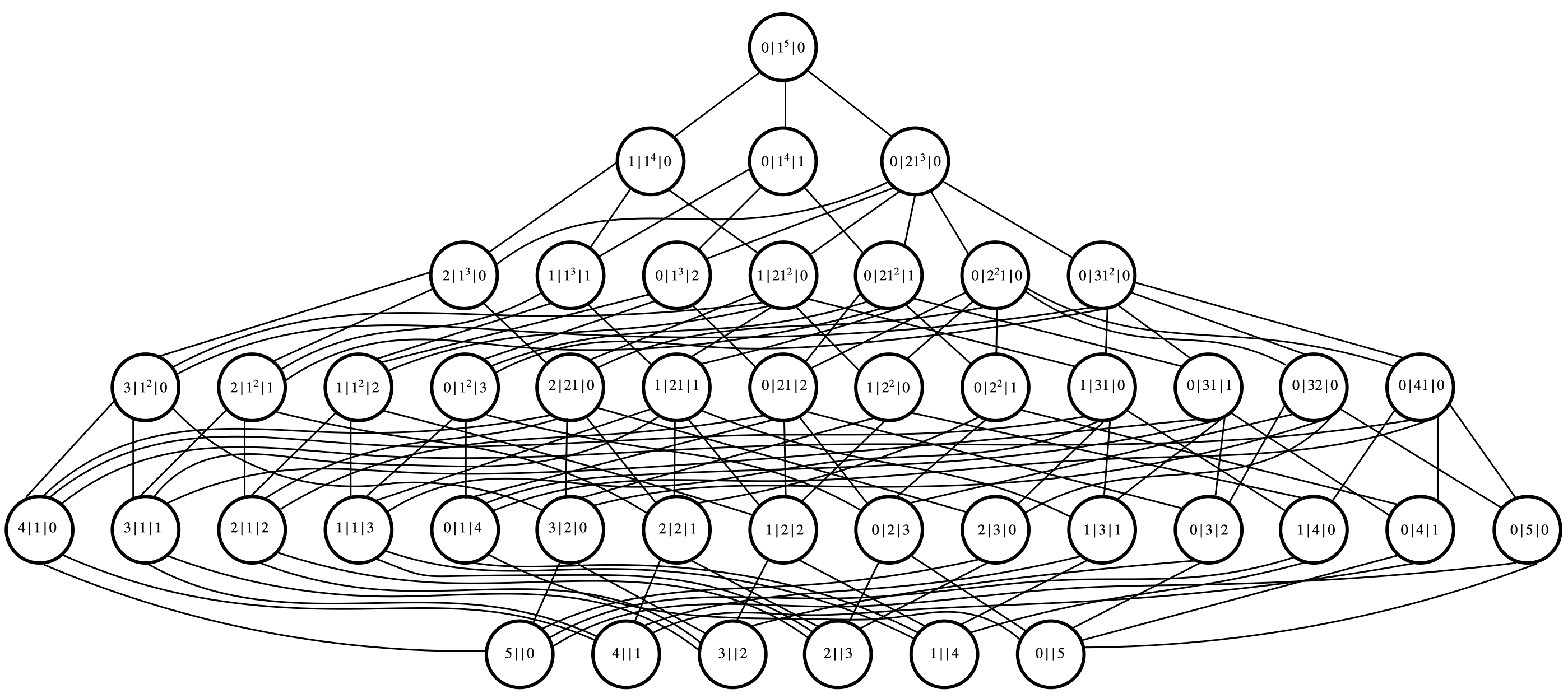}
    \caption{The Euler stratification of five points in $\mathbb{C}^*$.}
    \label{fig:strataC1}
\end{figure}

Much like the strata $S_\lambda$ in \eqref{eq:paramSlambda}, the stratum $S_{m_0|\lambda|m_\infty}$ is parametrized by 
\[
    (\P^1)^k \dashrightarrow S_{m_0|\lambda|m_\infty},\quad ((a_1:b_1),\ldots, (a_k: b_k)) \, \longmapsto  \, x_1^{m_0}(a_1x_0 + b_1x_1)^{\lambda_1}\cdots(a_kx_0 + b_kx_1)^{\lambda_k}x_0^{m_\infty}.
\]
Equations for its closure $\nabla_{m_0|\lambda|m_\infty}$ are thus easily deduced from equations for $\nabla_\lambda$. To define the partial ordering of the strata $S_{m_0|\lambda|m_\infty}$, let $(m_0, \lambda, m_\infty)$ be the unique partition of $n$ for which $S_{m_0|\lambda|m_\infty} \subseteq S_{(m_0,\lambda,m_\infty)}$. The following definition of ``$\prec$'' is natural: 
\begin{equation} \label{eq:order} m_0|\lambda|m_\infty \prec m_0' | \lambda' | m_\infty' \, \,  \Longleftrightarrow \, \, (m_0,\lambda,m_\infty) \prec (m_0', \lambda, m_\infty'), \, \, \, m_0 \leq m_0' \text{ and } \, m_\infty \leq m_\infty'. \end{equation}

This discussion leads to the following analog of Theorem \ref{thm:StratiRootsProj}.

\begin{theorem} \label{thm:strataXveryaffP1}
Let $(\X^*)^c$ be as in \eqref{eq:XveryaffP1} and consider the projection 
$\pi\colon (\X^*)^c\rightarrow Z=\P^n$. 
\begin{enumerate}
\setlength\itemsep{0em}
    \item[(i)] The set of strata $\left\{ S_{m_0|\lambda|m_\infty} \right\}$ ranging over all $m_0 + m_\infty \leq n$ and all partitions $\lambda$ of $n- m_0 - m_\infty$, ordered by \eqref{eq:order}, forms an Euler stratification of $\pi$.
    \item[(ii)] The Zariski closure $\nabla_{m_0|\lambda|m_\infty}$ of $S_{m_0|\lambda|m_\infty}$ is an irreducible projective variety. It equals the disjoint union $\nabla_{m_0|\lambda|m_\infty} = \coprod_{ m_0|\lambda|m_\infty \prec m_0'|\lambda'|m_\infty'} S_{m_0'|\lambda'|m_\infty' }$.
    \item[(iii)] For $z\in S_{m_0|\lambda|m_\infty}$, the Euler characteristic of the fiber is $\chi((\X^*)_z^c) = |\lambda|$. This equals the dimension $\dim S_{m_0|\lambda|m_\infty} = \dim \nabla_{m_0|\lambda|m_\infty} = |\lambda|$.
\end{enumerate}
\end{theorem}
As a consequence of Theorem \ref{thm:strataXveryaffP1}, the number of strata is $\sum_{i = 0}^n (n +1 - i){\cal P}(i)$, where ${\cal P}(i)$ is the number of partitions of $i$. This is large compared to the ${\cal P}(n)$ strata for \eqref{eq:XprojP1}.

To end the section, we replace $\mathbb{P}^1$ by a smooth projective curve $X\subset \P^N$ and consider families of points arising as the intersection of $X$ with a hyperplane $H_z$. Here, the index $z$ stands for the coefficients of the defining linear equation of $H$: $z\in Z=(\P^N)^{\vee}$. Our goal is to compute the Euler stratification of $Z$, i.e., the loci in $Z$ where $X\cap H_z$ consists of a constant number of points. Note that if $X \subset \mathbb{P}^n$ is the rational normal curve of degree $n$, then we recover the Euler stratification of \eqref{eq:XprojP1}. Recall that the projective dual variety $X^\vee \subseteq (\mathbb{P}^N)^\vee$ is the coordinate projection of the conormal variety ${\rm Con}(X) \subseteq \mathbb{P}^N \times (\mathbb{P}^N)^\vee$ to the dual projective space $(\mathbb{P}^N)^\vee$. Our algorithm is deduced from the following~result.

\begin{proposition}
    \label{prop:tevelevEuler}
    Let $X$ be an irreducible smooth projective curve such that the dual variety $X^{\vee}$ is a hypersurface. Let $H\in X^{\vee}$ be such that $X$ is not contained in $H$. Then
    \[
        \chi(X\cap H) = \deg X - \mult_H X^{\vee}.
    \]
\end{proposition}
Here, $\mult_H X^{\vee}$ denotes the multiplicity of $H \in X^{\vee}$. Proposition \ref{prop:tevelevEuler} is a consequence of Theorem \ref{thm:tevelev}, which summarizes \cite{dimca1986milnor} and \cite[Theorems 10.8 and 10.9]{tevelev2006projective}. 

\begin{theorem} \label{thm:tevelev}
    Let $X\subset \P^N$ be an irreducible smooth $d$-dimensional projective variety such that $X^{\vee}$ is a hypersurface. Let $H\in X^{\vee}$, let ${\rm pr} \colon \mathrm{Con}(X) \rightarrow X^\vee$ be the projection of the conormal variety, and let ${\cal L}$ be a line bundle on $X$.
    \begin{enumerate}
    \setlength\itemsep{-0.0cm}
        \item[(i)] If ${\rm pr}^{-1}(H)$ is finite,
        then the multiplicity of $X^{\vee}$ at $H$ is given by
    \[
        \mult_H X^{\vee} = \sum_{p\in X \cap H} \mu(X\cap H, p).
    \]
    \item[(ii)] For a global section $s\in H^0(X,\L)\setminus \{0\}$ of ${\cal L}$ with zero locus $V\subset X$, we have
    \[
        \mu(V) = (-1)^d(\chi(V) - \chi(X,\L)),
    \]
    where $\chi(X,\L)$ is the Euler characteristic of the zero locus of a generic section.
    \end{enumerate}
\end{theorem}

\begin{proof}[Proof of Proposition \ref{prop:tevelevEuler}]
    Choose $s\in \O_X(1)$ so that $X\cap H = s^{-1}(0) = V$. For $X$ one-dimensional, we have $\chi(X, \O_X(1)) = \deg X$ and, using Theorem \ref{thm:tevelev}(ii), we obtain
    \begin{equation}
        \label{equ:MilnorEulerDeg}
        \mu(X\cap H) = \sum_{p\in X\cap H} \mu(X\cap H, p) = \deg X - \chi(X\cap H).
    \end{equation}
    Combining \eqref{equ:MilnorEulerDeg} with Theorem \ref{thm:tevelev}(i), we arrive at the desired statement.
\end{proof}

Proposition \ref{prop:tevelevEuler} gives rise to Algorithm \ref{alg:stratifyPointsCurve}. The computation of $\Delta_X$ in line 1 can be performed using standard symbolic variable elimination techniques or numerical alternatives. We illustrate the algorithm in the following example. 

\begin{algorithm}[t] 
\small
  \caption{Euler stratification of $X\cap H_z$}\label{alg:stratifyPointsCurve}
  \textbf{Input:} homogeneous generators $f_1, \ldots, f_\ell \in \mathbb{Q}[x_0,\ldots,x_{N}]$ of the vanishing ideal of $X = V(f_1, \ldots, f_\ell)$ satisfying the assumptions of Proposition \ref{prop:tevelevEuler} and an integer $1\leq k\leq \deg X$\\ 
  \textbf{Output:} ideal $I\in \mathbb{Q}[z_0,\ldots,z_{N}]$ such that: $z\in V(I)$ generic $\Rightarrow \chi(X\cap H_z) = \deg X - k$
  \begin{algorithmic}[1]
    \State $\Delta_X \gets$ defining equation of $X^{\vee}$ in $\mathbb{Q}[X_0,\ldots ,X_{N}]$
    \State $\mathfrak{m} \gets$ maximal ideal generated by $\left\langle X_0 - z_0, \ldots, X_{N} - z_{N} \right\rangle \subset \mathbb{Q}(z_0,\ldots,z_{N})[X_0,\ldots ,X_{N}]$
    \State $Q \gets$ quotient ring $\mathbb{Q}(z_0,\ldots,z_{N})[X_0,\ldots ,X_{N}] / \mathfrak{m}^k$
    \State $M \gets$ companion matrix representing multiplication by $\Delta_X$ in $Q$
    \State \Return ideal generated by entries of $M$
  \end{algorithmic}
\end{algorithm}

\begin{example}
    Let $X = V(f) \subset \P^2$ be the elliptic curve defined by $f = x_0^3 - x_0x_2^2 - x_1^2x_2 +x_2^3$. The output of Algorithm \ref{alg:stratifyPointsCurve} for $k=1$ is the defining equation of the dual curve $X^{\vee}$:
    \[
        4z_0^6 + 4z_0^5z_2 + z_0^4z_1^2 + 36z_0^3z_1^2z_2 - 4z_0^3z_2^3 + 18z_0^2z_1^4 + 
        30z_0^2z_1^2z_2^2 + 24z_0z_1^4z_2 - 23z_1^6 + 54z_1^4z_2^2 - 27z_1^2z_2^4.
    \]
    This vanishes on $z \in (\mathbb{P}^2)^{\vee}$ for which $H_z = \{z_0 x_0 + z_1 x_1 + z_2 x_2 = 0\}$ intersects $X$ in at most two points. For $k=2$, the algorithm returns an ideal whose radical is generated by 
    \[ \begin{matrix}
         z_0^3 - 9z_0z_1^2 - 3z_0z_2^2 - 6z_1^2z_2, &  z_0^2z_2^2 - 50z_0z_1^2z_2 - 18z_0z_2^3 + 28z_1^4 - 63z_1^2z_2^2, \\
        7z_0^2z_1z_2 + 56z_0z_1^3 + 18z_0z_1z_2^2 + 45z_1^3z_2 - 9z_1z_2^3, 
        & 28z_0^2z_1^2 + 9z_0^2z_2^2 + 54z_0z_1^2z_2 + 6z_0z_2^3 + 21z_1^2z_2^2.
        \end{matrix}
    \]
    This defines nine points $z \in (\mathbb{P}^2)^{\vee}$ for which $H_z$ intersects $X$ in a single point.
\end{example}

\section{Projective hypersurfaces} \label{sec:4}

We switch to the hypersurface setting. Throughout the section, $F(x,z) \in \mathbb{C}[z][x]$ is a bihomogeneous polynomial in variables $x = (x_0, \ldots, x_d)$ and parameters $z = (z_0, \ldots, z_m)$. We develop an algorithm for computing the Euler stratification of the family
\begin{equation} \label{eq:Xproj2} 
\pi_F: {\cal X}_F \rightarrow Z, \quad \text{where} \quad {\cal X}_F \, = \, \{ (x,z) \in \mathbb{P}^d \times Z \, : \, F(x,z) \neq 0 \}.\end{equation}
We assume for simplicity that $Z$ is a closed irreducible subvariety of $\mathbb{P}^m$. The first task is to determine for which $z \in Z$ the Euler characteristic of $\X_{F,z} = \pi_F^{-1}(z)$ differs from the generic value. We make this precise by defining the \emph{Euler discriminant} of our family. We point out that this term was coined by Esterov in \cite[Section 1.2]{esterov2013discriminant} to mean something slightly different. 
\begin{definition} \label{def:eulerdisc}
    By Proposition \ref{prop:Xproj}, there exists an integer $\chi^*$ and a dense open Euler~stratum $S^* \subseteq Z$ such that $\chi(\X_{F,z}) = \chi^*$ for all $z \in S^*$. The \emph{Euler discriminant variety} of $\pi_F: {\cal X}_F \rightarrow Z$ is the Zariski closure of $Z \setminus S^*$ in $Z$. We denote this variety by $\nabla_\chi(\pi_F)$, or simply $\nabla_\chi$.
\end{definition}

\begin{algorithm}[H] 
\small
  \caption{Euler stratification of \eqref{eq:Xproj2}}\label{alg:stratify_proj}
  \textbf{Input:} bihomogeneous polynomial $F\in \mathbb{Q}[x,z]$ and the defining ideal $I_Z$ of $Z$\\ 
  \textbf{Output:} prime ideals of closed strata in the Euler stratification of $\pi_F:{\cal X}_F \rightarrow Z$
  \begin{algorithmic}[1]
    \State $\text{strata}_0 \gets \{ I_Z \} $ 
    \State $i \gets 0$
    \While{$\text{strata}_i \neq \{ \mathbb{Q}[z] \} $} \Comment{iteration ends when computed strata are empty}
    \State $i \gets i + 1$ 
    \State $\text{strata}_i \gets \{ \} $\Comment{Initialize an empty list of strata}
        \For{$I \in \text{strata}_{i-1}$}
      \State $I(\nabla_\chi) \gets \text{Euler discriminant ideal of } {\cal X}_F \cap (\mathbb{P}^d \times V(I))$
      \State $\text{strata}_i \gets \text{strata}_i \cup \{\text{minimal primes of } I(\nabla_\chi)\}$
      \EndFor
    \EndWhile\label{euclidendwhile}
    \State \textbf{return} $\bigcup_{i} \text{strata}_i$
  \end{algorithmic}
\end{algorithm}

Computing an Euler stratification comes down to an iterated computation of Euler discriminants. This is clarified by Algorithm \ref{alg:stratify_proj}, which uses the notation $I_Z = I(Z)$ for the vanishing ideal of $Z \subseteq \mathbb{P}^m$. The algorithm computes the Euler discriminant for $z \in V(I_Z) = Z$ first. After that, it restricts the family to each of the irreducible components of $\nabla_\chi$ to find deeper strata, and so on. Notice that one can easily adapt the algorithm to keep track of the partial order relations between strata.  The important step in the algorithm is Line 7, the computation of $\nabla_\chi$. The rest of the section is devoted to that step. 

Our algorithm uses a formula from \cite{huh2012milnor} which expresses $\chi(D(F))$ in terms of the cohomology class of the graph of the Gauss map $\nabla F$. This is the closure $\Gamma = \overline{\Gamma^\circ}$ of
\begin{equation} 
    \label{eq:GammaJ} 
    \Gamma^\circ \, = \, \{ (x, y) \in D(F) \times (\mathbb{P}^d)^{\vee} \, : \, y = \nabla F(x) \}   
\end{equation}
in $\mathbb{P}^d \times (\mathbb{P}^d)^{\vee}$. In what follows, let $H_k, H_k^{\vee}$ be generic $k$-planes in $\mathbb{P}^d$ and $(\mathbb{P}^d)^{\vee}$ respectively, and let $[H_k \times H_\ell^{\vee}]$ be the generators of the Chow ring of $\mathbb{P}^d \times (\mathbb{P}^d)^{\vee}$. 
We state a version of Huh's result \cite[Theorem 9]{huh2012milnor} and outline a geometric proof.

\begin{theorem} \label{thm:Huh}
    Let $F \in \mathbb{C}[x_0, \ldots, x_d]$ be any non-constant homogeneous polynomial, and let $D(F) = \{ x \in \mathbb{P}^d \, : \, F(x) \neq 0 \}$ be its hypersurface complement. Suppose that
    \[ [\Gamma] \, = \, \sum_{i = 0}^d e_i \, [H_{d-i} \times H_{i}^{\vee}] \quad \in \, A_d(\mathbb{P}^d \times (\mathbb{P}^d)^{\vee})\]
    is the class of $\Gamma$ from \eqref{eq:GammaJ} in the Chow ring of $\mathbb{P}^d \times (\mathbb{P}^d)^{\vee}$. Setting $H_{-1} = \emptyset$, we have
    \[  \chi(D(F)) \, = \, \sum_{i=0}^d \chi((D(F) \cap H_i) \setminus (D(F)\cap H_{i-1})) \, = \, \sum_{i=0}^d (-1)^i \deg \nabla F_{|H_i} \, = \, \sum_{i=0}^{d}(-1)^i \, e_i . \]
\end{theorem}

\begin{proof}
    The first equality follows from excision. The second equality is Theorem~\ref{thm:Dimca-Papadima}: 
    \[ \deg \nabla F_{|H_i} \, = \, (-1)^i \, \chi((D(F) \cap H_i) \setminus (D(F) \cap H_{i-1})).\]
    It remains to show that $\deg \nabla F_{|H_i}$ equals $e_i$. That is, $\deg \nabla F_{|H_i}$ equals the number of intersection points in $\Gamma \cap (H_i \times H_{d-i}^{\vee})$. The linear space $H_i$ can be represented as the projectivized row span of an $(i+1) \times (d+1)$ matrix $\mathbb{L}_i$. Let $\tilde{x}_0, \ldots, \tilde{x}_i$ be coordinates on $H_i$, so that $x= \mathbb{L}_i^t \cdot \tilde{x}$. We set $F_{|H_i} = F(\mathbb{L}_i^t \cdot \tilde{x})$, and the chain rule implies that 
    \[ \nabla F_{|H_i}(\tilde{x}) \, = \,  \left ( 
        \frac{\partial F_{|H_i}}{\partial \tilde{x}_0} : \cdots : \frac{\partial F_{|H_i}}{\partial \tilde{x}_i}
 \right ) \, = \, \mathbb{L}_i \left ( \restr{\left ( 
        \frac{\partial F}{\partial x_0} : \cdots : \frac{\partial F}{\partial x_d}
    \right )}{x = \mathbb{L}_i^t \cdot \tilde{x}} \right). \]
    Here, $\mathbb{L}_i(y)$ denotes the projection $\mathbb{L}_i: (\mathbb{P}^d)^{\vee} \dashrightarrow (\mathbb{P}^i)^{\vee}$ represented by the matrix $\mathbb{L}_i$.
    This establishes a bijection between the set $\{\tilde{x} \in D(F) \cap H_i  \, : \, \nabla F_{|H_i}(\tilde{x}) = \tilde{y}\}$~and 
    \[\{ (x,y) \in \Gamma \, : \, x \in H_i, \, \,  \mathbb{L}_i (y) = \tilde{y} \} \, = \, \Gamma \cap (H_i \times H_{d-i}^{\vee}) \]
    for generic $\tilde{y} \in (\mathbb{P}^i)^{\vee}$.  The condition $\mathbb{L}_i(y) = \tilde{y}$ determines a $(d-i)$-dimensional linear space $H_{d-i}^{\vee}$ in which $y$ is contained. The intersection $\Gamma \cap (H_i \times H_{d-i}^{\vee})$ is transverse, so that 
    \[ \deg \nabla F_{|H_i} \, = \, |\{\tilde{x} \in D(F) \cap H_i  \, : \, \nabla F_{|H_i}(\tilde{x}) = \tilde{y}\}| \, = \, | \Gamma \cap (H_i \times H_{d-i}^{\vee}) | \, = \, e_i. \qedhere  \]
\end{proof}

The numbers $e_i$ can be read from the coefficients of the (bi-graded) Hilbert polynomial of $\Gamma$, see \cite[Section 2]{huh2012milnor}. They can also be computed numerically, without computing the ideal of $\Gamma$, by intersecting $\Gamma^\circ$ from \eqref{eq:GammaJ} with products of random linear spaces $H_{i} \times H_{d-i}^{\vee}$. The number of intersection points is $e_i$. Alternatively, one computes the cardinality of a generic fiber of $\nabla F_{|H_i} : D(F) \cap H_i \rightarrow \mathbb{P}^i$. We have implemented these approaches in Julia. For this, we rely on the numerical homotopy continuation algorithms implemented in \cite{breiding2018homotopycontinuation}.

\begin{example} \label{ex:whitneycountercontd}
    We consider the cuspidal plane curves from Example \ref{ex:counterwhitney1}.
    For $z = (1:1)$, our implementation computes the following information, which is easily verifiable:
    \[ e_0 \, = \, \deg \nabla F_{|H_0} \, = \,  1, \quad e_1 \, = \,  \deg \nabla F_{|H_1} \, = \,  2, \quad e_2 \, = \,  \deg \nabla F_{|H_2} \, = \,  2. \qedhere \]
\end{example}

By Theorem \ref{thm:Huh}, the Euler discriminant $\nabla_\chi(\pi_F)$ is contained in the union of loci where the degrees of the restrictions of the Gauss map $\nabla F(x,z)_{|H_i}$ drop. Let $n^*$ be the degree of $\nabla F(x,z)$ for generic $z \in Z$. We define the \emph{polar discriminant} of a nonzero bi-homogeneous polynomial $F \in \mathbb{C}[z][x]$ as the Zariski closure of $\{ z \in Z \, : \, {\rm deg} \, \nabla F(x,z) < n^* \}\subset Z$. In what follows, we propose a randomized algorithm for computing the polar discriminant. 

Let $b \in (\mathbb{P}^d)^{\vee}$ be a generic point. The degree of $\nabla F$ is the number of points $x \in D(F)$ satisfying $\nabla F(x) = b$. Adding parameters $z \in Z$, we consider the incidence variety 
\begin{equation} \label{eq:Ybcirc} {\cal Y}^\circ_b \, = \, \left \{ (x,z) \in {\cal X}_F \, : \, {\rm rank} \begin{pmatrix}
    \frac{\partial F(x,z)}{\partial x_0}& \frac{\partial F(x,z)}{\partial x_1} & \cdots & \frac{\partial F(x,z)}{\partial x_d} \\ 
    b_0 & b_1 & \cdots & b_d
\end{pmatrix} \, \leq  \, 1  \right \}.\end{equation}
Here, ${\cal X}_F$ is as in \eqref{eq:Xproj2}. The closure of ${\cal Y}^\circ_b$ in $\mathbb{P}^d \times Z$ is denoted by ${\cal Y}_b$. By assumption, the projection $\pi_Z: {\cal Y}_b \rightarrow Z$ has generically finite fibers of cardinality $n^*$. The polar discriminant is contained in the image of ${\cal Y}_b \setminus {\cal Y}_b^\circ = {\cal Y}_b \cap V(F(x,z))$ under that projection. 

 \begin{proposition} \label{prop:strictcontainedproj}
     Let ${\cal X}_F$ be the family from \eqref{eq:Xproj2}, with $Z \subseteq \mathbb{P}^m$ any irreducible quasi-projective variety. Let ${\cal Y}_b^\circ$ be as in \eqref{eq:Ybcirc}, where $b \in \mathbb{P}^d$ is generic. The variety ${\cal Y}_b \cap V(F)$ has dimension at most $\dim Z-1$. Hence, $\pi_Z({\cal Y}_b \cap V(F)) \subsetneq Z$ is a strict containment.
 \end{proposition}
 \begin{proof}
     The following incidence variety is irreducible of dimension $d + \dim Z$:
     \begin{equation} \mathcal{Y}^\circ \, = \,  \{ (x,z, b) \in {\cal X}_F \times \mathbb{P}^d_b \, : \, \nabla F(x,z) = b  \}. \end{equation}
     This is seen from the obvious parametrization ${\cal X}_F \rightarrow {\cal Y}^\circ$. A general fiber of the dominant map $\mathcal{Y}^\circ \rightarrow \mathbb{P}^d_b$ is of pure dimension $\dim Z$. That fiber is ${\cal Y}_b^\circ$, so its boundary in the closure ${\cal Y}_b = \overline{{\cal Y}_b^\circ}$ has dimension at most $\dim {\cal Y}_b - 1 = \dim Z - 1$. 
 \end{proof}

The containment of the polar discriminant in $\pi_Z({\cal Y}_b \cap V(F))$ might be strict. For instance, there might be components which depend on the specific choice of $b$, see Example \ref{ex:polardiscriminant}. We can discard such spurious loci by repeating this procedure for several values of $b$ and intersecting the results. Here, $\dim Z + 1 \leq m + 1$ random choices for $b$ suffice. The discussion is summarized in Algorithm \ref{alg:polardisc}. Lines 3 and 4 can be executed using standard Gr\"obner basis techniques.

\begin{algorithm}[H] 
\small
  \caption{Polar discriminant of $\nabla F(x,z) $}\label{alg:polardisc}
  \textbf{Input:} bihomogeneous polynomial $F\in \mathbb{Q}[x,z]_n$ and the defining ideal $I_Z$ of $Z$\\ 
  \textbf{Output:} minimal primes of the ideal defining the polar discriminant of $F$
  \begin{algorithmic}[1]
  \For{$i = 0, \ldots, \dim Z$} 
    \State $b_i \gets $ a random point in $(\mathbb{P}_{\mathbb{Q}}^d)^{\vee}$
    \State compute the vanishing ideal of the closure ${\cal Y}_{b_i}$ of ${\cal Y}_{b_i}^\circ$ in \eqref{eq:Ybcirc}
    \State $I_i \gets$ the defining ideal of the projection $\pi_Z ({\cal Y}_{b_i} \cap V(F(x,z)))$
    \EndFor
    \State \Return minimal primes of $I_0 + I_1 + \cdots + I_{\dim Z }$
  \end{algorithmic}
\end{algorithm}

\begin{example} \label{ex:polardiscriminant}
The polar discriminant of the ternary quadric in \eqref{eq:Fconics} 
is \eqref{eq:conicdiscriminant}. It is computed efficiently with our code, see \eqref{eq:mathrepo}, choosing uniformly random integers $b_i$ between $-100$ and 100, a total number of $\dim Z + 1 = 6$ times. Each of the ideals $I_i$ in Algorithm \ref{alg:polardisc} decomposes as $\langle \Delta \rangle \cap \langle G(z,b) \rangle$, where $G(z,b)$ defines a hypersurface in $Z = \mathbb{P}^5$ which depends on $b$. 
\end{example}

\begin{algorithm}
\small
  \caption{Upper bound for the Euler discriminant of $\pi_F: {\cal X}_F \rightarrow Z $}\label{alg:eulerdiscproj}
  \textbf{Input:} bihomogeneous polynomial $F\in \mathbb{Q}[z][x]$ and the defining ideal $I_Z$ of $Z$\\ 
  \textbf{Output:} minimal prime ideals of a variety containing the Euler discriminant of $\pi_F: {\cal X}_F \rightarrow Z $ 
  \begin{algorithmic}[1]
  \For{$i = 0, \ldots, d-1$}
   \State $I_i \gets \langle 0 \rangle$
    \For{$j = 0, \ldots, \dim Z$}
    \State $H_{i,j} \gets $ random $i$-plane in $\mathbb{P}^d$ 
    \State $I_i \gets I_i \, + $ ideal of the polar discriminant of $(\nabla F_{|H_{i,j}}, Z)$
    \EndFor 
    \EndFor
    \State $I_d \gets $ polar discriminant of $(\nabla F, Z)$
    \State \Return minimal primes of $I_0 \cap I_1 \cap \cdots \cap I_d$
  \end{algorithmic}
\end{algorithm}

The polar discriminant characterizes for which $z \in Z$ the degree of $\nabla F$ drops. By Theorem \ref{thm:Huh}, for computing the Euler discriminant, we must also check the degrees of the restrictions $\nabla F_{|H_i}$ for generic $i$-planes $H_i \subseteq \mathbb{P}^d$. The polar discriminant of $\nabla F_{|H_i}: D(F) \cap H_i \rightarrow \mathbb{P}^i$ may depend on $H_i$. Similarly as above, we can eliminate unwanted contributions by computing the polar discriminant repeatedly, for $m+1$ choices of $H_i$. This happens in lines 1--7 in Algorithm \ref{alg:eulerdiscproj}. The union of all these polar discriminants may contain the Euler discriminant strictly (see Example \ref{ex:counterwhitney3}). Nonetheless, by Proposition \ref{prop:strictcontainedproj}, using Algorithm \ref{alg:eulerdiscproj} in line 7 of Algorithm \ref{alg:stratify_proj} leads to a correct but possibly redundant Euler stratification.

\begin{example} \label{ex:counterwhitney3}
    We have seen in Example \ref{ex:counterwhitney1} that the Euler discriminant for $F(x,z) = z_0x_0^3+z_1x_1^2x_2$ is given by $z_0 = 0$. The fiber of $\nabla F$ over $(b_0:b_1:b_2)$ consists of the two points 
    \[ \left ( \sqrt{\frac{b_0}{3z_0}} : \sqrt{\frac{b_2}{z_1}} : \frac{b_1}{2z_1}\sqrt{\frac{z_1}{b_2}} \right ), \quad \left ( \sqrt{\frac{b_0}{3z_0}} : -\sqrt{\frac{b_2}{z_1}} : -\frac{b_1}{2z_1}\sqrt{\frac{z_1}{b_2}} \right ). \]
    When $(z_0:z_1) \rightarrow (0:1)$, these points approach $(1:0:0)$. The point $(1:0:0) \times (0:1)$ is indeed contained in ${\cal Y}_b \cap V(F(x,z))$. However, a similar argument shows that $(1:0)$ is also contained in the polar discriminant of $\nabla F$. However, $(1:0)$ does not belong to the Euler discriminant: the alternating sum $\sum_{i=0}^{2}(-1)^i e_i$ for $z = (1:0)$ equals one, as for generic $z$.
\end{example}

Example \ref{ex:counterwhitney3} raises the question to what extent the polar and Euler discriminant agree. For plane curve families with isolated singularities, it turns out that they coincide. 
\begin{proposition}
    \label{prop:PolarEqEuler}
    Let $F(x_0,x_1,x_2, z)$ be a squarefree bihomogeneous polynomial. If $Z \subseteq \mathbb{P}^m$ is such that the degree of the (reduced) plane curve ${\cal X}_{F,z}^c \subset \mathbb{P}^2$ is constant for all $z$, then the Euler discriminant of ${\cal X}_F \rightarrow Z$ coincides with the polar discriminant of $F$. 
\end{proposition}
\begin{proof}
    It follows from Theorem \ref{thm:Dimca-Papadima} that for any $z, z' \in Z$, we have 
    \[ \chi({\cal X}_{F,z}) - \chi({\cal X}_{F,z'}) \, = \, \deg \nabla F(x,z) - \deg \nabla F(x,z') - ( \deg {\cal X}_{F,z}^c - \deg {\cal X}_{F,z'}^c). \qedhere\]
\end{proof}

\section{Very affine hypersurfaces} \label{sec:5}
Our original motivation for studying Euler stratifications comes from particle physics and statistics. In these applications, the families of interest are those from Equation \eqref{eq:Xveryaff}:
\begin{equation} \label{eq:Xveryaff2}
\mathcal{X}^*_f \, = \, \{ (x, z) \in T \times Z \, : \, f(x,z) \neq 0 \} \, = \, {\cal X}_F, \quad \text{where $F(x,z) = x_0 \cdots x_d \, f(x,z)$.}
 \end{equation} 
As in the Introduction, $T = D(x_0 \cdots x_d)$ is the dense torus of $\mathbb{P}^d$.
In physics, the signed Euler characteristic of ${\cal X}_{F,z}$ measures the number of \emph{master integrals}, which form a basis of a vector space generated by Feynman integrals \cite{agostini2022vector,bitoun2019feynman}. The dimension of that vector space depends on kinematic parameters, here captured by $z$. In algebraic statistics, very affine hypersurface complements arise as \emph{discrete exponential families}, also called \emph{toric models} \cite{amendola2019maximum}. The general result \cite[Theorem 20]{catanese2006maximum} implies that the maximum likelihood (ML) degree of such models equals the signed Euler characteristic of ${\cal X}_{F,z}$. In this context, $z$ are model~parameters. Euler stratifications aim to understand completely how the ML degree depends on $z$. 

We start by considering families in which the monomial support of $f$ is fixed, and all coefficients vary freely. More precisely, we fix a matrix $A$ with nonnegative integer entries of size $(d+1) \times (m+1)$ with no repeated columns, and with constant column sum $n>0$. In Example \ref{ex:conicsveryaff}, the parameters are $d = 2$, $m = 5$ and $n = 2$. The columns of $A$ are the integer vectors $a_0, \ldots, a_m \in \mathbb{N}^{d+1}$. They serve as exponents in the formula for $f(x,z)$: 
\begin{equation} \label{eq:fGKZ} f(x,z) \, = \, z_0 \, x^{a_0} \, + \, z_1 \, x^{a_1} \, + \, \cdots \, + \, z_m \, x^{a_m}. \end{equation}
The Euler discriminant from Definition \ref{def:eulerdisc} of the family ${\cal X}_F \rightarrow Z = \mathbb{P}^m$ is well-understood. We summarize this in Theorem \ref{thm:GKZesterov} below, which needs some more notation. The convex hull ${\rm Conv}(A)$ of the columns of $A$ is a polytope in $\mathbb{R}^{d+1}$ of dimension at most $d$. Its lattice volume is denoted by ${\rm vol}(A)$. The \emph{$A$-discriminant variety} $\nabla_A$ is the closure of all parameter values $z^*$ for which $V(f(x,z^*)) \cap T$ is singular. If $\nabla_A$ has codimension one, then its defining equation is denoted by $\Delta_A$ (this is defined up to a nonzero scalar multiple). If $\nabla_A$ has codimension at least two, we set $\Delta_A = 1$. The \emph{principal $A$-determinant} is defined as the following $A$-resultant:
    \[ E_A(z) \, = \, {\rm Res}_A \left ( f, \,  x_1 \, \frac{\partial f}{\partial x_1}, \, \ldots,\,  x_d \, \frac{\partial f}{\partial x_d} \right ) \quad \in  \mathbb{Q}[z_0, \ldots, z_m]. \]
For more background, see \cite[Chapter 10]{gelfand2008discriminants}. The polynomial $E_A(z)$ is known to factor into a product of discriminants. For each face $Q \preceq {\rm Conv}(A)$, let $A_Q$ be the submatrix of $A$ consisting of columns which lie on $Q$. By \cite[Chapter 10, Theorem 1.2]{gelfand2008discriminants}, we have 
\begin{equation} \label{eq:EAfactors} E_A(z) \, = \, \prod_{Q \preceq {\rm Conv}(A)} \Delta_{A_Q}^{e_Q}. \end{equation}
Here $e_Q \geq 1$ is the multiplicity of the projective toric variety of $A$ along its torus orbit corresponding to $Q$. The next theorem summarizes \cite[Theorem 13]{amendola2019maximum} and \cite[Theorem 2.36]{esterov2013discriminant}. 

\begin{theorem} \label{thm:GKZesterov}
    The signed Euler characteristic $(-1)^d \cdot \chi({\cal X}_{f,z}^*) = (-1)^d \cdot \chi({\cal X}_{F,z})$ with $f(x,z)$ as in \eqref{eq:fGKZ} is ${\rm vol}(A)$ if and only if $E_A(z) \neq 0$. Moreover, for a generic point $z$ on a codimension-one discriminant $\nabla_{A_Q}$, we have ${\rm vol}(A) - (-1)^d \cdot \chi({\cal X}_{F,z}) = e_Q$, with $e_Q$ as in \eqref{eq:EAfactors}. 
\end{theorem}
\begin{example}
    The principal $A$-determinant $E_A$ from Example \ref{ex:conicsveryaff} has seven factors, one for each face of the triangle ${\rm Conv}(A)$. The toric variety of $A$ is the $2$-uple embedding of $\mathbb{P}^2$. Its multiplicity along each torus orbit is one. All exponents in \eqref{eq:EAfactors} are one, and the Euler characteristic differs from its generic value by one on each component of $\{E_A(z) = 0 \}$.
\end{example}

Before switching back to the case where $Z$ is any irreducible subvariety of $\mathbb{P}^m$, we discuss useful formulae for $\chi({\cal X}_{F,z^*}) = \chi(D(F(x,z^*)))$ when $z = z^*$ is fixed and arbitrary. In \cite[Theorem 1]{huh2013maximum}, Huh expresses the signed Euler characteristic of $D(F)$ as the number of critical points of a \emph{log-likelihood function}. We phrase this theorem in our notation. 

\begin{theorem} \label{thm:huh2}
    Let $F(x) = x_0 \cdots x_d \,  f(x) \neq 0$. There is a dense open subset $U \subseteq \mathbb{P}^d$ such that for $(\nu_0: \nu_1: \ldots: \nu_d) \in U$, the following equations have precisely $(-1)^d \cdot \chi(D(F))$ non-degenerate isolated solutions in $D(F) = T \setminus V(f)$: 
        \begin{equation} \label{eq:likelihood} \frac{\nu_1}{x_1} + \nu_0 \, \frac{\partial_1 f}{f} \, = \, \cdots \, = \, \frac{\nu_d}{x_d} + \nu_0 \, \frac{\partial_d f}{f} \, = \, 0.\end{equation} 
\end{theorem}

In Theorem \ref{thm:huh2}, a \emph{non-degenerate isolated solution} is a point $x \in D(F)$ satisfying \eqref{eq:likelihood} at which the Jacobian determinant of these $d$ equations does not vanish. The notation $\partial_i f$ is short for $\frac{\partial f}{\partial x_i}$. Using coordinates $t_i = \frac{x_i}{x_0}$ on $T$, one writes the equations \eqref{eq:likelihood} as \begin{equation} \label{eq:dlog0} {\rm dlog} \, f (1, t_1, \ldots, t_d )^{\nu_0} \, t_1^{\nu_1} \cdots t_d^{\nu_d}  \, = \, 0 . \end{equation}
Theorem \ref{thm:huh2} has important practical implications: it turns out that computing Euler characteristics of very affine hypersurfaces via critical points is efficient and reliable in practice, see e.g.~\cite{agostini2023likelihood,agostini2022vector,fourlectures}. Below, we will use Theorem \ref{thm:huh2} to compute Euler discriminants. 

Theorem \ref{thm:Dimca-Papadima} expresses the degree of $\nabla F$ as a sum of Euler characteristics: 
\begin{equation} \label{eq:decompDP}
    \deg \nabla F \, = \, (-1)^d \cdot \chi(D(F)) \, + \,  (-1)^{d-1} \cdot \chi(D(F) \cap H).
\end{equation}
We explain how this is a decomposition of critical points according to a \emph{likelihood degeneration}, in the sense of \cite{agostini2023likelihood}. First, we show that $\deg \nabla F$ is indeed a critical point count. 

\begin{proposition} \label{prop:ncrit}
    Let $0 \neq f \in \mathbb{C}[x_0, \ldots, x_d]_n$ be a form of degree $n$ and let $L = b_0 x_0 + \cdots + b_d x_d \in \mathbb{C}[x_0, \ldots, x_d]_1$ be a generic linear form. The degree of the Gauss map $\nabla F$ with $F = x_0 x_1 \cdots x_d \, f$ equals the number of solutions in $D(F\cdot L)$ to the equations 
    \begin{equation} \label{eq:critnabla} \frac{1}{x_i} + \frac{\partial_i f}{f} - (d+ n + 1) \, \frac{b_i}{L} \, = \, 0, \quad i = 1, \ldots, d.\end{equation}
\end{proposition}
\begin{proof}
    By definition, the degree of $\nabla F$ is the number of points $x \in D(F)$ satisfying 
    \[ {\rm rank} \begin{pmatrix}
        \frac{\partial F}{\partial x_0} & \frac{\partial F}{\partial x_1} & \cdots & \frac{\partial F}{\partial x_d} \\
        b_0 & b_1 & \cdots & b_d
    \end{pmatrix} \, = \, {\rm rank} \begin{pmatrix}
        \frac{1}{x_0} + \frac{\partial_0 f}{f} & \frac{1}{x_1} + \frac{\partial_1 f}{f}  & \cdots & \frac{1}{x_d} + \frac{\partial_d f}{f}  \\
        b_0 & b_1 & \cdots & b_d
    \end{pmatrix} \, = \, 1,\]
    for a generic point $b \in \mathbb{P}^d$. Here we divided the first row of the matrix by $F = x_0x_1 \cdots x_d \, f$, which does not change the rank. Multiplying from the right with the matrix obtained by replacing the first column of the identity matrix with $(x_0, \ldots, x_d)$, we obtain the equations 
    \[ {\rm rank} \begin{pmatrix}
        d + n + 1 & \frac{1}{x_1} + \frac{\partial_1 f}{f}  & \cdots & \frac{1}{x_d} + \frac{\partial_d f}{f}  \\
        b_0 x_ 0 + \cdots + b_d x_d & b_1 & \cdots & b_d
    \end{pmatrix} \, = \, 1. \]
    Note that this last step is allowed because $F$ has the special form $F = x_0x_1 \cdots x_d \, f$. 
    Since $b$ is generic, we find that $L(x) = b_0 x_0 + \cdots + b_d x_d \neq 0$, and we obtain \eqref{eq:critnabla}.
\end{proof}
In $t$-coordinates, the equations \eqref{eq:critnabla} are the critical point equations given by 
\begin{equation} \label{eq:dlog1} {\rm dlog} \,  f(1, t_1, \ldots, t_d) \, t_1 \cdots t_d  \, L(1,t_1,\ldots,t_d)^{-(d+n+1)} \, = \, 0. \end{equation}
The decomposition \eqref{eq:decompDP} is seen from Theorem \ref{thm:huh2} and Proposition \ref{prop:ncrit} by introducing a parameter $\varepsilon$ in the exponents, and then taking the limit $\varepsilon \rightarrow 0$. Explicitly, we set 
\[ {\cal L}(\varepsilon) \, = \, f(1,t_1,\ldots,t_d)^{(1-\varepsilon)\nu_0 + \varepsilon}t_1^{\nu_1(1-\varepsilon)+\varepsilon} \cdots t_d^{\nu_d(1-\varepsilon)+\varepsilon}  L(1,t_1,\ldots,t_d)^{-\varepsilon(d+n+1)}. \]
The solutions to ${\rm dlog} \, {\cal L}(\varepsilon) = 0$ are tuples of Puiseux series in $\varepsilon$. For $\varepsilon = 1$, they evaluate to the ($\deg \nabla F$)-many solutions to \eqref{eq:dlog1}. The equations are explicitly given by 
\begin{equation} \label{eq:epsiloneq} \frac{\nu_i (1-\varepsilon) + \varepsilon}{t_i} + ((1-\varepsilon)\nu_0 + \varepsilon) \, \frac{\partial_{t_i}\tilde{f}}{\tilde{f}} - \varepsilon (d+ n+ 1) \frac{b_i}{\tilde{L}} \, = \, 0, \quad i = 1, \ldots, d,  \end{equation}
where $\tilde{f}= f(1,t)$ and $ \tilde{L} = L(1,t)$. 
When $\varepsilon \rightarrow 0$, the solutions split up into two groups. The first group of solutions is of the form $t(\varepsilon) = t^{(0)}  + $ higher order terms in $\varepsilon$, with lowest order term $t^{(0)} \in (\mathbb{C}^*)^d$ satisfying $L(1, t^{(0)}) \neq 0$. Vanishing of the lowest order term in \eqref{eq:epsiloneq} implies that these solutions converge to the critical points of \eqref{eq:dlog0}. The leading term $t^{(0)} \in (\mathbb{C}^*)^d$ of the second group of solutions satisfies $L(1,t^{(0)}) = 0$, and $L(1, t(\varepsilon)) = \varepsilon L^{(0)} + $ higher order terms, with $L^{(0)} \in \mathbb{C}^*$. In particular, these solutions converge to points on $H = V(L)$. Plugging these Ans\"atze for $t$ and $L$ into \eqref{eq:epsiloneq} and setting the lowest order term to zero shows that $t^{(0)}$ are the critical points of the restriction of $t_1^{\nu_1}\cdots t_d^{\nu_d}\tilde{f}^{\nu_0}$ to $H$. By Theorem \ref{thm:huh2}, the first group of solutions consists of $(-1)^d \cdot \chi(D(F))$ points, and the second group has $(-1)^{d-1} \cdot \chi(D(F) \cap H)$ solutions. They sum up to the number of solutions to \eqref{eq:dlog1} by Equation \eqref{eq:decompDP}.

\begin{remark}
The discussion above amounts to solving the equations \eqref{eq:epsiloneq} \emph{tropically}. We have described the valuation of the coordinates $t_i, \tilde{L}$ on $D(L)$ for each $\varepsilon$-series solution. It is beyond the scope of this article to do this more rigorously. For more, see \cite[Section 7]{agostini2023likelihood}.
\end{remark}

\begin{example}
    We work out an example with $d = n = 1$. Let $f = x_1-x_0$, $L = x_1 - 2 x_0$. We are interested in the critical points of $\log {\cal L}(\varepsilon)$, given in the coordinate $t = x_1/x_0$ by 
    \[ {\cal L}(\varepsilon) \, = \, (t-1)^{\nu_0(1-\varepsilon) + \varepsilon}t^{\nu_1(1-\varepsilon) + \varepsilon} (t-2)^{-3\varepsilon}.\]
    For concreteness, we choose $\nu_0 = -11, \nu_1 = 17$. The solutions of \eqref{eq:epsiloneq} are 
    \[ t_{\pm}(\varepsilon) \, = \, \frac{29 - 27 \varepsilon \pm \sqrt{25 + 154 \varepsilon - 167 \varepsilon^2}}{2(6-7\varepsilon)}. \]
    The two solutions of \eqref{eq:dlog1} are $t_\pm(1) = -1 \mp \sqrt{3}$. When $\varepsilon \rightarrow 0$, the solutions converge to $t_+(0) = 17/6$ and $t_-(0) = 2$. Here $\{17/6\}$ is the unique critical point of \eqref{eq:dlog0} (it is a rational number since $|\chi(D(x_0x_1f))| = 1$), and $\{2 \}= V(L)$. Equation \eqref{eq:decompDP} reads $2 = 1 + 1$.  
\end{example}

Theorem \ref{thm:huh2} can be used to show that the Euler characteristic of very affine hypersurfaces is semicontinuous. This is false for general projective hypersurface families, see Example \ref{ex:conics_intro}.
\begin{theorem}
    \label{thm:Zsclosed}
    Let ${\cal X}_F$ be the family from \eqref{eq:Xveryaff2}, with $Z \subseteq \mathbb{P}^m$ any irreducible quasi-projective variety. The set $Z_{\leq k} = \bigcup_{q \leq k} Z_q$, with $Z_q = \{ z \in Z \, : \, (-1)^d \cdot \chi({\cal X}_{F,z}) = q \}$, is closed in~$Z$. 
\end{theorem}
\begin{proof}
    By Lemma \ref{lem:Xveryaff}, the set $Z_q$ is constructible. That is, $Z_q$ has a decomposition 
    \[ Z_q = (V^q_1 \setminus W^q_1) \cup \cdots \cup (V^q_{m_q} \setminus W^q_{m_q}),\] 
    where $V^q_i = \overline{V^q_i \setminus W^q_i}$ are irreducible closed subvarieties of $Z$. By Theorem \ref{thm:huh2}, the Euler characteristic of ${\cal X}_{F,z}$ for $z \in V^q_i$ is the generic number of non-degenerate isolated solutions to the following system of $d+1$ parametric polynomial equations on $(\mathbb{C}^*)^{d+1}$: 
    \[ \nu_1 \, x_1^{-1} + \nu_0 \, y \, \partial_1 f  \, = \, \cdots \, = \, \nu_d \, x_d^{-1} + \nu_0 \, y\, \partial_d f \, = \, yf - 1 \, = \, 0. \]
    By \cite[Theorem 7.1.4]{sommese2005numerical}, this quantity is semicontinuous, and the maximum $q$ is attained for generic $\nu$ and $z$ in the dense open subset $V^q_i \setminus W^q_i$. We conclude that $V_i^q \subseteq Z_{\leq q} \subseteq Z_{\leq k}$. As a consequence, $\overline{Z_q} \subseteq Z_{\leq k}$, and we are done. This proof appeared in \cite[Theorem 3.1]{fevola2024principal}.
\end{proof}

Algorithm \ref{alg:eulerdiscproj} can in principle be used to compute the Euler discriminant of the family \eqref{eq:Xveryaff2}. However, there is a more efficient (but still randomized) algorithm based on Theorem \ref{thm:huh2}. We fix random parameters $\nu \in \mathbb{P}^d$ and, much like in \eqref{eq:Ybcirc}, we define the incidence 
\begin{equation} \label{eq:Ynucirc}
{\cal Y}_\nu^\circ \, = \, \left \{ (x,z) \in {\cal X}_F \, : \, \frac{\nu_i}{x_i} + \nu_0 \, \frac{\partial_i f(x,z)}{f(x,z)} \, = \, 0, \, i = 1, \ldots, d \right \}.
 \end{equation}
 The closure in $\mathbb{P}^d \times Z$ is denoted by ${\cal Y}_\nu$. The number of critical points of \eqref{eq:dlog0} drops when a solution is contained in the boundary ${\cal Y}_\nu \setminus {\cal Y}_\nu^\circ$. Hence, the Euler discriminant is contained in the projection $\pi_Z({\cal Y}_\nu \setminus {\cal Y}_\nu^\circ) = \pi_Z({\cal Y}_\nu \cap V(F))$. This is always strictly contained in $Z$.

 \begin{proposition} \label{prop:strictcontained}
     Let ${\cal X}_F$ be the family from \eqref{eq:Xveryaff2}, with $Z \subseteq \mathbb{P}^m$ any irreducible quasi-projective variety. Let ${\cal Y}_\nu^\circ$ be as in \eqref{eq:Ynucirc}, where $\nu \in \mathbb{P}^d$ is generic. The variety ${\cal Y}_\nu \cap V(F)$ has dimension at most $\dim Z-1$. Hence, $\pi_Z({\cal Y}_\nu \cap V(F)) \subsetneq Z$ is a strict containment.
 \end{proposition}
 \begin{proof}
    The proof is similar to that of Proposition \ref{prop:strictcontainedproj}. In this case, the incidence variety is
     \begin{equation} \label{eq:Ycirc} \mathcal{Y}^\circ \, = \, \left \{ (x,z, \nu) \in {\cal X}_F \times \mathbb{P}^d_\nu \, : \, \frac{\nu_i}{x_i} + \nu_0 \, \frac{\partial_i f(x,z)}{f(x,z)} \, = \, 0, \, i = 1, \ldots, d \right \}. \end{equation}
     It is again parametrized by ${\cal X}_F$, and hence irreducible of dimension $d + \dim Z$.
 \end{proof}

\begin{algorithm}[t] 
\small
  \caption{Euler discriminant of $\pi_{x_0x_1 \cdots x_d f}: {\cal X}_{x_0x_1 \cdots x_d f} \rightarrow Z $}\label{alg:eulerdiscveryaff}
  \textbf{Input:} bihomogeneous polynomial $f\in \mathbb{Q}[x,z]$ and the defining ideal ${\cal I}_Z$ of $Z$\\ 
  \textbf{Output:} minimal prime ideals of the Euler discriminant of $\pi_{x_0x_1 \cdots x_d f}: {\cal X}_{x_0x_1 \cdots x_d f} \rightarrow Z $
  \begin{algorithmic}[1]
    \For{$i = 0, \ldots, \dim Z$}
        \State $\nu \gets $ a random point in $\mathbb{P}^d$
        \State compute the vanishing ideal of the closure ${\cal Y}_\nu$ of ${\cal Y}_\nu^\circ$ in \eqref{eq:Ynucirc}
        \State $I_i \gets $ the defining ideal of the projection $\pi_Z({\cal Y}_\nu \cap V(F))$
    \EndFor
    \State \Return minimal primes of $I_0 + I_1 + \cdots + I_{\dim Z}$
  \end{algorithmic}
\end{algorithm}

Algorithm \ref{alg:eulerdiscveryaff} is based on ideas similar to those in Section \ref{sec:4}. Components in $\pi_Z({\cal Y}_\nu \cap V(F))$ which depend on $\nu$ are eliminated by performing the computation $\dim Z + 1$ times, and adding up the ideals. Importantly, in practice we observe that this repeated computation \emph{is not needed}. This suggests that the projection of $\overline{{\cal Y}^\circ} \cap (V(F) \times \mathbb{P}^d_\nu) \subset \mathbb{P}^d_x \times Z \times \mathbb{P}^d_\nu$ to $Z \times \mathbb{P}^d_{\nu}$ only has components of the form $V \times \mathbb{P}^d_\nu$ or $Z \times W$. Here ${\cal Y}^\circ$ is as in \eqref{eq:Ycirc}. The components of the form $V \times \mathbb{P}^d_\nu$ contribute to the Euler discriminant. Among the components of the form $Z \times W$ there are the hyperplanes identified in \cite{sattelberger2023maximum}. See Section \ref{sec:7}, Problem 5 for future work.
 
\begin{example}
Our Julia implementation of Algorithm \ref{alg:eulerdiscveryaff} is used as in Algorithm \ref{alg:stratify_proj} to compute an Euler stratification of ${\cal X}_f^*$, with $f$ the ternary quadric from \eqref{eq:Fconics}. As predicted by \eqref{eq:EAconics}, there are seven closed strata of dimension four. Among the deeper strata, there are 21 of dimension three, 27 of dimension two, and 15 of dimension one.
The latter consist of twelve lines
\[ \begin{matrix}
    \langle z_5, z_2, z_1, z_0\rangle, & 
 \langle z_5, z_3, z_2, z_0\rangle, & 
 \langle z_5, z_4, z_2, z_0\rangle, & 
 \langle z_3, z_2, z_1, z_0\rangle, & 
 \langle z_4, z_3, z_2, z_0\rangle, & 
 \langle z_4, z_3, z_1, z_0\rangle, \\
 \langle z_5, z_3, z_1, z_0\rangle, & 
 \langle z_5, z_4, z_1, z_0\rangle, & 
 \langle z_5, z_4, z_3, z_0\rangle, &
 \langle z_5, z_3, z_2, z_1\rangle, & 
 \langle z_5, z_4, z_3, z_2\rangle, & 
 \langle z_5, z_4, z_3, z_1\rangle, 
\end{matrix}\]
and three quadric curves: $\langle 4z_3z_5 - z_4^2, z_2, z_1, z_0\rangle,
 \langle 4 z_0 z_5 - z_2^2, z_4, z_3, z_1\rangle, 
 \langle 4 z_0 z_3 - z_1^2, z_5, z_4, z_2\rangle $.
\end{example}

\vspace{-0.4cm}

\section{Case studies} \label{sec:6}

In this section we present several case studies of Euler stratifications. We compute the full stratification of binary octics in Section \ref{sec:61} ($n = 8$ in \eqref{eq:XprojP1}). In Section \ref{sec:62}, we recall the relation between the Euler stratification of bilinear forms and the matroid stratification of the Grassmannian. Sections \ref{sec:63} and \ref{sec:64} discuss the applications in physics and statistics respectively. Finally, in Section \ref{sec:65}, we stratify hyperplane sections of Hirzebruch surfaces. Our code uses the packages \texttt{Oscar.jl} (v1.0.4) \cite{OSCAR}  and \texttt{HomotopyContinuation.jl} (v2.9.4)~\cite{breiding2018homotopycontinuation}.

\vspace{-0.2cm}

\subsection{Binary octics} \label{sec:61}

We come back to binary forms and their coincident root loci as discussed in Section \ref{sec:3}. Our goal is to compute the full Euler stratification for a binary octic ($n = 8$). According to Theorem \ref{thm:StratiRootsProj}, this requires the computation of ideals $I(\nabla_{\lambda})$ of coincident root loci for partitions $\lambda$ of up to eight. By Theorem \ref{thm:strataXveryaffP1}, the Euler stratification of eight points in $\mathbb{C}^*$ is easily deduced from the same ideals. We describe our method of computing $I(\nabla_\lambda)$.\par 
First, we compute the degrees of a minimal generating set of $I(\nabla_{\lambda})$. In \cite[Table 1]{lee2016duality}, these numbers are reported for partitions of up to seven. We extend this table to the case $\lambda \vdash 8$ using finite field computations, see Table \ref{table}.
\begin{table}
\begin{center}
    \caption{Degrees of generators for the ideal of coincident root loci of a binary octic.}
    \label{table}
    \begin{tabular}{cc|cc} 
    \toprule 
    $\lambda$ & generators of $I(\nabla_{\lambda})$ & $\lambda$ &  generators of $I(\nabla_{\lambda})$\\
    \midrule 
    (8) & $2^{28}$ & & \\
    (7,1) & $2^{15}$ & (4,4) & $3^{74}$\\
    (6,1,1) & $2^6,~3^1$ & (3,1,1,1,1,1) & $6^1,~7^6,~8^6$\\
    (6,2) & $2^6,~3^1,~4^{46}$ & (3,2,1,1,1) & $6^1,~7^6,~8^{6},~10^{46}$\\
    (5,1,1,1) & $2^1,~3^6,~4^6$ & (3,2,2,1) & $6^1,~7^6,~8^{546}$\\
    (5,2,1) & $2^1,~3^6,~4^6,~6^{54}$ & (3,3,1,1) & $5^{84}$\\
    (5,3) & $2^1,~3^{19},~4^{64}$ & (3,3,2) & $4^1,~5^{166},~6^{100}$\\
    (4,1,1,1,1) & $4^5,~5^9,~6^1$ & (2,1,1,1,1,1,1) & $14^1$\\
    (4,2,1,1) & $4^6,~5^9,~6^1,~8^{54}$ & {\color{gray} (2,2,1,1,1,1)} & {\color{gray} $11^{19}$}\\
    (4,2,2) & $4^{15},~5^{221}$ & (2,2,2,1,1) & $8^{120}$\\
    (4,3,1) & $4^{45}$ & (2,2,2,2) & $5^{286}$\\
    \bottomrule 
    \end{tabular} 
\end{center}
\end{table}
Here, the entry for $\lambda = (2^2,1^4)$ is conjectural; the finite field computation did not terminate within seven days (see Section \ref{sec:7}, Problem 1). Using the parametrization of $\nabla_{\lambda}$ given by \eqref{eq:paramSlambda}, we sample points on $\nabla_{\lambda}$ and interpolate with polynomials of the computed degrees using linear algebra over $\qq$. In this process, we exploit the homogeneity of $I(\nabla_{\lambda})$ with respect to the bigrading given by the exponent matrix 
\[
    \begin{pmatrix}
        8 & 7 & \dots & 1 & 0 \\
        0 & 1 & \dots & 7 & 8 \\
    \end{pmatrix}.
\]
That grading is obtained from the $\C^*$-action on the space of binary octics $\P^8$. It divides the set of degree $d$ monomials in $\C[z_0,\dots,z_8]$ into $8d+1$ many buckets of monomials with the same bidegree. The resulting linear system that needs to be solved becomes substantially smaller. For example, the maximal number of monomials of degree eleven in one bucket is 2430, compared to 75582 monomials overall. The implementation as well as a database with all ideals for the Euler stratification of a binary octic in $\C^*$ can be found at \eqref{eq:mathrepo}. 

\subsection{Matroid stratification of bilinear forms} \label{sec:62}
We start with very affine hypersurfaces defined by bilinear forms. For $d_1, d_2 \geq 1$, consider
\begin{equation}  \label{eq:fbilinear} f(x,y,z) \, = \, \sum_{i= 0}^{d_1} \sum_{j = 0}^{d_2} z_{ij} \, x_i \, y_j.  \end{equation}
The parameter space $Z = \mathbb{P}^{(d_1+1)(d_2+1)-1}$ is that of $(d_1 + 1) \times (d_2 + 1)$ matrices $z = (z_{ij})_{i,j}$. Since $f$ is bilinear, it is natural to consider its zero locus in the torus of $\mathbb{P}^{d_1} \times \mathbb{P}^{d_2}$. In this subsection, we set $(\mathbb{C}^*)^{d_1+d_2} \simeq T \subset \mathbb{P}^{d_1} \times \mathbb{P}^{d_2}$ and we consider the following family:
\[ {\cal X}_f^* \, = \, \{ (x,y,z) \in T \times Z \, : \, f(x,y,z) \neq 0 \}.\]
The fibers ${\cal X}_{f,z}^*$ of ${\cal X}_f^* \rightarrow Z$ represent independence models in algebraic statistics. Their Euler characteristic was studied in \cite{clarke2023matroid}.
One of the main results \cite[Theorem 1.3]{clarke2023matroid} states that for any $z \in Z$, the signed Euler characteristic $(-1)^{d_1 + d_2} \cdot \chi({\cal X}_{f,z}^*)$ equals the M\"obius invariant of the rank-$(d_1 + 1)$ matroid represented by the $(d_1+1) \times (d_1 + d_2 + 2)$ matrix $[{\rm id}_{d_1+1} \, \, \, z]$.

A corollary is that an Euler stratification of ${\cal X}_f^* \rightarrow Z$ is induced by the matroid stratification of the Grassmannian ${\rm Gr}(d_1+1,d_1+d_2+2)$. In particular, the Euler discriminant is the union of all hypersurfaces defined by the $(d_1+1) \times (d_1+1)$ minors of $[{\rm id}_{d_1+1} \, \, \, z]$. Equivalently, these are the hypersurfaces defined by all minors of the matrix $z$. 

If $z \in \mathbb{R}^{(d_1+1) \times (d_2+1)}$ has real entries, then the M\"obius invariant $(-1)^{d_1+d_2} \cdot \chi({\cal X}^*_{f,z})$ equals the number of bounded cells in the following hyperplane arrangement complement:
\[ \left \{ t \in \mathbb{R}^{d_1} \, : \, t_1 \cdots t_{d_1} \, \left (z_{00} +\sum_{i=1}^{d_1} z_{i0} t_i \right ) \cdots \left (z_{0d_2}+ \sum_{i=1}^{d_1} z_{id_2} t_i \right ) \neq 0 \right \}. \]

\begin{example}
    For $d_1 = 2, d_2 = 1$, the polynomial $f$ and the matrix $[{\rm id}_{d_1+1} \, \, \, z]$ are given by 
    \[ \begin{matrix} f \, = \,  z_{00} \, x_0y_0 + z_{10}\,  x_1y_0 + z_{20} \, x_2y_0  \\ \quad \quad + \, z_{01} \,  x_0y_1 + z_{11} \, x_1 y_1 + z_{21} \,  x_2 y_1, \end{matrix} \quad \quad \quad [{\rm id}_{d_1+1} \, \, \, z] \, = \, \begin{pmatrix}
        1 & 0 & 0 & z_{00} & z_{01} \\
        0 & 1 & 0 & z_{10} & z_{11} \\ 
        0 & 0 & 1 & z_{20} & z_{21}
    \end{pmatrix}.\]
    For generic choices of $z_{ij} \in \mathbb{R}$, the lines $\{z_{0i} + z_{1i}t_1 + z_{2i} t_2 = 0\}, i = 0, 1$, together with $\{t_1 = 0 \}$ and $\{t_2 = 0\}$, define three bounded cells in $\mathbb{R}^2$.  The Euler discriminant of ${\cal X}_f^*$ vanishes when at least one of these cells collapses. It is given by the product of the minors of~$z$:  \[\Delta_\chi \, = \,  z_{00}z_{10}z_{20}z_{01}z_{11}z_{21}(z_{00}z_{11}-z_{01}z_{10})(z_{00}z_{21}-z_{01}z_{20})(z_{10}z_{21}-z_{11}z_{20}). \qedhere\]
\end{example}

\subsection{Feynman integrals} \label{sec:63}

Let ${\cal X}_{f}^*$ be the family of very affine hypersurfaces from \eqref{eq:Xveryaff2}. In particle physics, the signed Euler characteristic of the fiber ${\cal X}^*_{f,z}$ counts the number of \emph{master integrals} \cite{bitoun2019feynman}. These are (regularized) Feynman integrals which form a basis for the \emph{twisted cohomology} of ${\cal X}^*_{f,z}$ with respect to the ${\rm dlog}$ form \eqref{eq:dlog0}, see \cite{bitoun2019feynman,mastrolia2019feynman}. We consider Feynman integrals in parametric representation \cite[Section 2.5.4]{weinzierl}. In that context, the polynomial $f$ is the \emph{graph polynomial} associated to a Feynman graph. The variables $x$ are the integration variables, called \emph{Schwinger parameters}, and the parameters $z$ are \emph{kinematic data}, such as momenta and masses of fundamental particles. The theory of twisted cohomology extends to the more general class of \emph{Euler integrals} studied in~\cite{agostini2022vector,fourlectures}, where $f$ can be any polynomial.

Since $x$ is integrated out, a Feynman integral represents a multivalued function of $z$. The goal of \emph{Landau analysis} is to detect branch and pole loci of this multivalued function \cite{mizera2022landau}. By \cite[Theorem 56]{brown2009periods}, these loci are contained in the union of codimension-one strata in a Whitney stratification of ${\cal X}_f^* \rightarrow Z$. Brown defines the \emph{Landau variety} as this union \cite[Definition 54]{brown2009periods}. Algorithms for computing the Landau variety can be found in \cite{helmer2024landau,panzer2015algorithms}. The paper \cite{fevola2024principal} conjectures that the branch and pole loci are detected by a drop in the number of master integrals, i.e., by the Euler discriminant of the family ${\cal X}_{f}^*\rightarrow Z$. However, no algorithm for computing the Euler discriminant is presented in \cite{fevola2024principal}. Instead, an approximation of $\nabla_\chi$ is computed, called the \emph{principal Landau determinant}, using efficient techniques from numerical algebraic geometry inspired by \cite{mizera2022landau}. Here, we use Algorithm \ref{alg:eulerdiscveryaff} to compute the Euler discriminant of several Feynman integrals, including some appearing in \cite{helmer2024landau}. In all examples we checked, the Euler discriminant coincides with Brown's Landau variety. In future work, we aspire to develop numerical versions of our algorithms to tackle more challenging diagrams.

\begin{figure}
    \centering
    \includegraphics[width=0.75\linewidth]{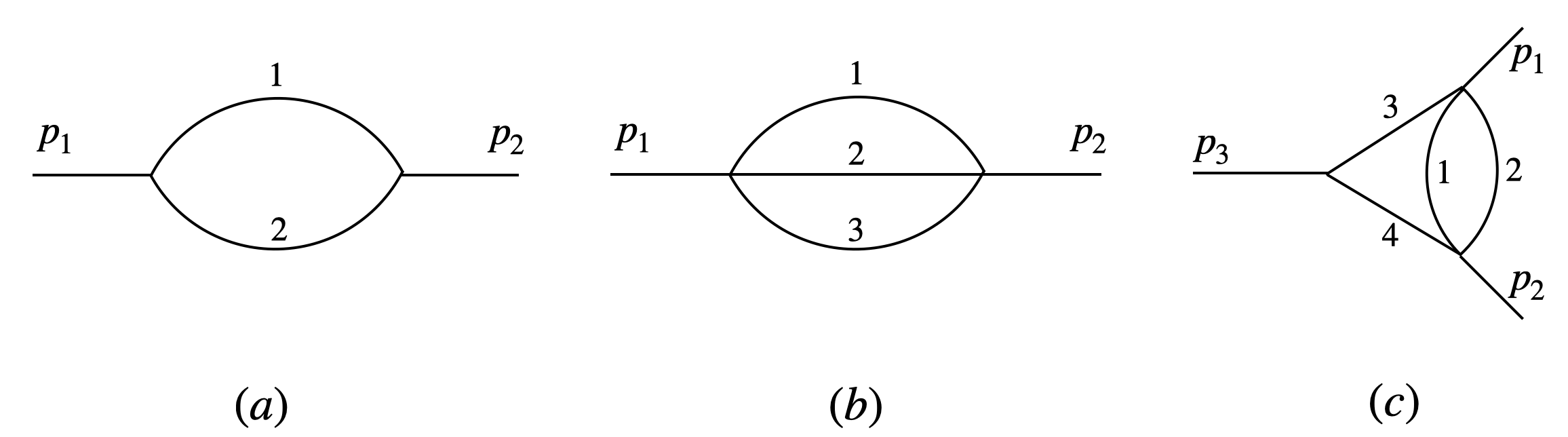}
    \caption{Three Feynman diagrams: bubble, sunrise and parachute. }
    \label{fig:feynman}
\end{figure}
\begin{example}
    We start with the bubble graph (Figure \ref{fig:feynman}(a)), whose graph polynomial is 
    \[ f \, = \, (x_0 - m_1x_1 - m_2 x_2)(x_1+x_2) + s \, x_1 x_2.\]
    The variety ${\cal X}_{f}^* \subset \mathbb{P}^2 \times Z$ is a family of curve complements over $Z = \mathbb{C}^3$, where $Z$ has coordinates $(m_1,m_2,s)$. The parameters represent the squared masses $m_i$ of particles traveling along the internal edges of the diagram, and the Lorenz invariant (or \emph{Mandelstam invariant}) $s$ depending on momenta of external particles. Algorithm \ref{alg:stratify_proj} computes the following strata:
    \[\begin{matrix} {\rm codim} \, 1: &  \langle  m_1 \rangle, \langle m_2 \rangle ,  \langle s \rangle, \langle m_1^2 - 2m_1m_2 - 2m_1s + m_2^2 - 2m_2s + s^2 \rangle
    \\
    {\rm codim} \, 2: & \langle m_2 - s, m_1 \rangle , \langle s, m_1\rangle , \langle m_2, m_1\rangle, \langle s, m_1 - m_2\rangle, \langle m_2, m_1 - s\rangle, \langle s, m_2 \rangle \\
    {\rm codim} \, 3: & \langle m_1, m_2, s \rangle
    \end{matrix}
    \]
    This Euler stratification agrees with the Whitney stratification in \cite[Section V, \S A]{helmer2024landau}.
\end{example}

\begin{example}
    Next, we consider the sunrise graph, shown in part (b) of Figure \ref{fig:feynman}. The (homogenized) graph polynomial $f$ defines a family of surface complements over $Z = \mathbb{C}^4$:
    \[ f \, = \, (x_0-m_1x_1 -m_2x_2 - m_3 x_3)(x_1x_2 + x_1x_3 + x_2x_3) + s \, x_1x_2x_3.
    \]
    The Euler discriminant agrees with the principal Landau determinant from \cite[Example~3.7]{fevola2024principal}: 
    \[ \begin{matrix}  \Delta_\chi \, = \, m_1 \cdot m_2 \cdot m_3 \cdot s \cdot (m_1^4 - 4 m_1^3 m_2 - 4 m_1^3 m_3 - 4 m_1^3 s + 6 m_1^2 m_2^2 + 4 m_1^2 m_2 m_3  \\ \quad \quad \quad + \, 4 m_1^2 m_2 s  + 6 m_1^2 m_3^2 + 4 m_1^2 m_3 s + 6 m_1^2 s^2 - 4 m_1 m_2^3 + 4 m_1 m_2^2 m_3 + 4 m_1 m_2^2 s \\ \quad \quad \quad + 4 m_1 m_2 m_3^2   - 40 m_1 m_2 m_3 s + 4 m_1 m_2 s^2 - 4 m_1 m_3^3 + 4 m_1 m_3^2 s + 4 m_1 m_3 s^2  \\ \quad \quad \quad   - \, 4 m_1 s^3 + m_2^4  - 4 m_2^3 m_3 - 4 m_2^3 s + 6 m_2^2 m_3^2 + 4 m_2^2 m_3 s + 6 m_2^2 s^2 - 4 m_2 m_3^3 \\ \quad \quad \quad + \, 4 m_2 m_3^2 s + 4 m_2 m_3 s^2 - 4 m_2 s^3 + m_3^4 - 4 m_3^3 s  + 6 m_3^2 s^2 - 4 m_3 s^3 + s^4).
    \end{matrix} \qedhere  \]
\end{example}

\begin{example}
    The graph polynomial of the parachute diagram (Figure \ref{fig:feynman}(c)) reads
    \[ f \, = \, (x_0 - \sum_{i=1}^4 m_i x_i)\cdot ((x_1+x_2)(x_3+x_4) + x_1x_2) + x_1x_2(M_1x_3 + M_2 x_4) + M_3 x_3x_4(x_1+x_2),\]
    where $M_i$ is the squared mass of the $i$-th external particle. With the specialization $m_1 = 1, m_4 = 2, M_1 = -1$ and all other parameters except $M_3$ equal to zero, we compute that the Euler discriminant agrees with the Landau variety in \cite[Section V, \S C]{helmer2024landau}: $\Delta_\chi = M_3(M_3-1)(M_3+1)(M_3-2)(M_3+2)$. In particular, the Euler discriminant contains the components which are missing in the principal Landau determinant, see \cite[Equation (3.18)]{fevola2024principal}. These components lie on a component of the Landau variety identified in \cite[Equation~(6.15)]{berghoff2022hierarchies}.
\end{example}

\subsection{Maximum likelihood estimation for toric models} \label{sec:64}

Let $(\mathbb{C}^*)^d \simeq T \subset \mathbb{P}^d$ be the dense torus of $\mathbb{P}^d$. The $(d+1) \times (m+1)$-matrix $A$ from Section \ref{sec:5} encodes a projective toric variety $X_A \subset \mathbb{P}^m$. It is obtained as the image closure of the monomial map $\phi_A: T \rightarrow \mathbb{P}^m$ given by $x \mapsto (x^{a_0}: \cdots : x^{a_m})$. For fixed $z \in \mathbb{P}^m$, the \emph{scaled toric variety} $z \star X_A$ is obtained from $X_A$ by applying the diagonal scaling $X_A \ni y \mapsto (z_0y_0 : \cdots: z_my_m)$. In algebraic statistics, the intersection of the scaled toric variety $z \star X_A$ with the $m$-simplex $\mathbb{P}^m_{>0}$ represents a \emph{discrete exponential family} or \emph{log-linear model} \cite{amendola2019maximum}. The maximum likelihood (ML) degree for this model is the signed Euler characteristic of $(z \star X_A) \setminus {\cal H}$, where ${\cal H} = V(y_0y_1\cdots y_m(y_0 + y_1 + \cdots + y_m))$ and $y_i$ are coordinates on $\mathbb{P}^m$. Reparametrizing if necessary, we may assume that the monomial map $\phi_A$ is one-to-one. Then the restriction of $\phi_A$ to ${\cal X}_{f,z}^*$, with $f$ as in \eqref{eq:fGKZ}, is an isomorphism ${\cal X}_{f,z}^* \simeq (z \star X_A) \setminus {\cal H}$. It follows that the Euler stratification of ${\cal X}_{f}^* \rightarrow \mathbb{P}^m$ from \eqref{eq:Xveryaff2}, with $f$ as in \eqref{eq:fGKZ}, completely describes the dependency of the ML degree of the discrete exponential family on the model parameters $z$. 

Notice that $\phi_A$ identifies $({\cal X}_{f,z}^*)^c \subset T$ with the hyperplane section $H_z \cap {\rm im} \, \phi_A \subset H_z \cap X_A$, where $H_z = V(z_0y_0 + \cdots + z_my_m)$. Since $\chi({\cal X}_{f,z}^*) = - \chi(({\cal X}_{f,z}^*)^c)$, we are stratifying $(\mathbb{P}^m)^\vee$ according to the Euler characteristic of $H_z$ intersected with the dense torus of $X_A$.

\begin{example} Following \cite{amendola2020maximum,amendola2023likelihood}, we consider matrices $A$ whose columns are lattice points of reflexive polygons. The corresponding toric surfaces are the Gorenstein Fano toric surfaces, of which there exist 16 isomorphism classes \cite[Figure 1]{amendola2020maximum}. Algorithm \ref{alg:stratify_proj} computes the Euler stratification for the first five reflexive polygons using Algorithm \ref{alg:eulerdiscveryaff} for the Euler discriminant. These are the polygons appearing in \cite[Table 2]{amendola2020maximum}. The results are summarized in Figure \ref{fig:reflexive}. The figure contains the polygons, the corresponding $A$-matrices, and a list of closed, irreducible strata in $\mathbb{P}^m$. The notation for a closed stratum $\overline{S}$ is $(\dim \overline{S}, \deg \overline{S}, \chi({\cal X}_{f,z}^*))^e$, where $z \in S$ is a generic point on the stratum and the exponent $e$ indicates how often a stratum of that type occurs. The generators for the defining ideal of each stratum and the code for reproducing these numbers can be found online at \eqref{eq:mathrepo}. Our results can be used to compute the \emph{ML degree drop} from \cite[Definition 2.7]{amendola2023likelihood} for any member of the respective model families.
\end{example}

\begin{figure}
    \centering
    \includegraphics[width=0.88\linewidth]{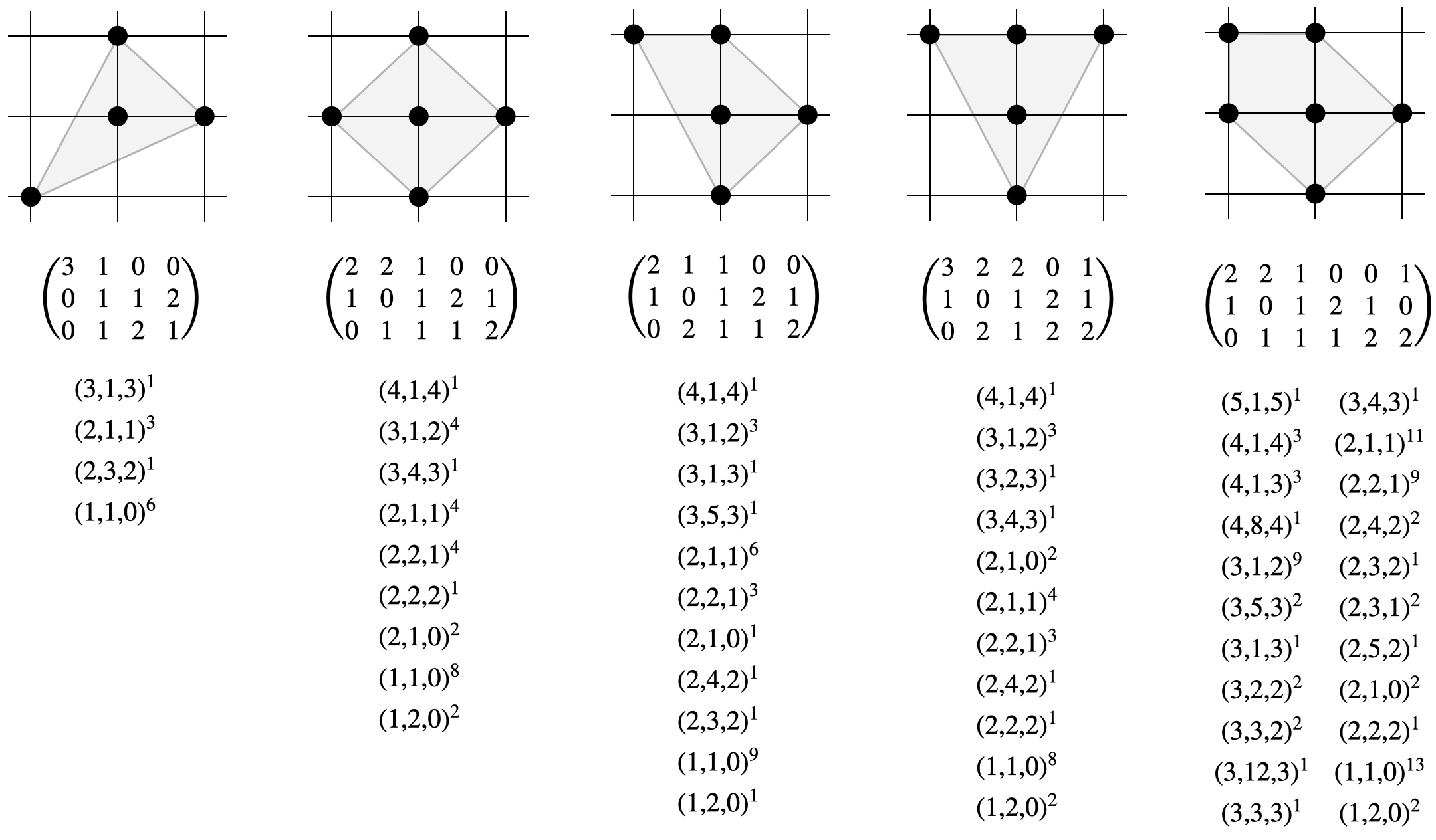}
    \caption{Euler stratifications for hyperplane sections of five toric Fano surfaces. }
    \label{fig:reflexive}
\end{figure}

\subsection{Hirzebruch surfaces} \label{sec:65}

A Hirzebruch surface is a toric surface $X_A \subset \mathbb{P}^{n_0 + n_1 +1}$ obtained from a matrix of the form
\[ A \, = \, \begin{pmatrix}
    n_1+1 & n_1 & \cdots & n_1-n_0+1 & n_1 & n_1-1 & \cdots & 0\\
    0 & 1 & \cdots & n_0 & 0 & 1 & \cdots & n_1 \\
     0 & 0 & \cdots & 0 & 1 & 1 & \cdots & 1  
\end{pmatrix}, \]
where $n_0 \leq n_1$ are positive integers.
As explained in Section \ref{sec:64}, a hyperplane section $H_z \cap {\rm im} \, \phi_A$ is isomorphic to a curve in the torus $(\mathbb{C}^*)^2 \simeq T \subset \mathbb{P}^{2}$. That curve is defined by
\[ f \, = \, z_{00}x_0^{n_1+1} + z_{10}x_0^{n_1}x_1 + \cdots + z_{n_00} x_0^{n_1-n_0+1}x_1^{n_0} + z_{01}x_0^{n_1}x_2 + z_{11}x_0^{n_1-1} x_1x_2 + \cdots + z_{n_11} x_1^{n_1}x_2. \]
This gives rise to a family of curves $\X^c_{f} \subset T \times \P^{n_0+n_1+1}$. The following Lemma already appeared as \cite[Theorem 17]{amendola2019maximum}. Here, we give a geometric proof of the statement. 
\begin{lemma} \label{lem:hirz}
    Let $f$ be as above. For any $z \in \mathbb{P}^{n_0+n_1+1}$, the signed Euler characteristic of $(\X^*_{f,z})^c=\{ x \in T\, : \, f(x;z) = 0 \}$ equals the number of roots $t \in \mathbb{C}^*$ of
    \begin{equation} \label{eq:lemhirz} (z_{00} + z_{10} t + z_{20}t^2 + \cdots + z_{n_00} t^{n_0})(z_{01} + z_{11} t + z_{21}t^2 + \cdots + z_{n_11} t^{n_1} ) = 0.\end{equation}
\end{lemma}
\begin{proof}
    Fix $z \in \mathbb{P}^{n_0 + n_1 + 1}$. Let $ (t_1, t_2) = (x_0^{-1}x_1, x_0^{-1}x_2)$ be coordinates on $T$. The curve $({\cal X}_{f,z}^*)^c$ is the zero locus of $x_0^{-(n_1+1)}f(x;z) = g_0(t_1;z) + t_2 \cdot g_1(t_1;z)$. Here the polynomials $g_0, g_1$ are precisely the factors appearing in Equation \eqref{eq:lemhirz}. Let $\pi_1: (\X^*_{f,z})^c \rightarrow \mathbb{C}^*$ be the projection to the $t_1$-coordinate. Let $U \subset \mathbb{C}^*$ be the open subset of points $t_1$ satisfying $g_0(t_1;z)g_1(t_1;z) \neq 0$. The restriction of $\pi_1$ to $\pi_1^{-1}(U)$ is a fibration whose fibers consist of a single point. The complement $\mathbb{C}^* \setminus U$ is a disjoint union $U_0 \sqcup U_1 \sqcup U_{01}$, where
\begin{align*}
        U_0  &= \{ t_1 \in \mathbb{C}^*\, : \, g_0(t_1;z) = 0, g_1(t_1;z) \neq 0 \},  \\
        U_1  &= \{ t_1 \in \mathbb{C}^*\, : \, g_0(t_1;z) \neq 0, g_1(t_1;z) = 0 \}, \\ 
        U_{01}  &= \{ t_1 \in \mathbb{C}^*\, : \, g_0(t_1;z) = g_1(t_1;z) = 0 \}.
        \end{align*}
    The fiber of $\pi_x$ over any point of $U_0 \sqcup U_1$ is empty, and the fiber over any point of $U_{01}$ is $\mathbb{C}^*$. Using the excision property and multiplicativity along fibrations of $\chi(\cdot)$, we obtain
    \[ 
        \chi((\X^*_{f,z})^c) \, = \, \chi(U) \cdot \chi(\{{\rm pt}\}) + \chi(U_0) \cdot \chi(\emptyset) + \chi(U_1) \cdot \chi(\emptyset) + \chi(U_{01}) \cdot \chi(\mathbb{C}^*) = \chi(U).
    \]
    Using excision once more we see that $\chi(({\cal X}_{f,z}^*)^c) = - \chi(\{t_1 \in \mathbb{C}^* \, : \, g_0(t_1;z)g_1(t_1;z) = 0 \})$. 
\end{proof}
Lemma \ref{lem:hirz} reduces the Euler stratification of $\X^*_{f}$ to the case $d=1$ as in Section \ref{sec:3}.
\begin{theorem} \label{thm:hirz}
    Let $\pi_f\colon \X^*_{f} \rightarrow \P^{n_0+n_1+1}$ be a family of curves with $f$ as above. There exists an Euler stratification of $\pi_f$ with $\sum_{i=0}^{n_0 + n_1} (n_0 + n_1 + 1 - i) {\cal P}(i) + 1$ many strata, where $\mathcal{P}(i)$ is the number of partitions of $i$. All but one of these strata correspond bijectively to the strata in the Euler stratification  of $\mathbb{P}^{n_0 + n_1}$ induced by a univariate polynomial of degree $n_0 + n_1$ on $\mathbb{C}^*$. 
\end{theorem}
\begin{proof}
    Consider the surjective morphism $\rho\colon \P^{n_0+n_1+1} \setminus E \rightarrow \P^{n_0 + n_1}$ sending the coefficients of $f(x;z) = f_0(x_0,x_1,z) + x_2f_1(x_0,x_1,z)$ to the coefficients of $f_0f_1$. Here, $E$ is the base locus $E = \{z_{00}=\dots = z_{n_00}=0\}\cup \{z_{01}=\dots = z_{n_11}=0\}$. By Lemma \ref{lem:hirz}, the preimages of the Euler strata of $f_0f_1$ in $\P^{n_0 + n_1}$ under $\rho$ give an Euler stratification of $\P^{n_0+n_1+1} \setminus E$. The remaining observation is that the Euler characteristic of $\X^*_{f,z}$ along the closed set $E$ is constant and equal to zero. The number of strata is a consequence of Theorem \ref{thm:strataXveryaffP1}.
\end{proof}

\section{Future directions} \label{sec:7}

We end with eight open questions/research problems related to Euler stratifications. 
\begin{enumerate}
\setlength \itemsep{0em}
    \item As indicated in Section \ref{sec:3}, not much is known in general about the defining ideals of coincident root loci of binary forms. Understanding these better with the help of modern tools in tropical geometry, numerical algebraic geometry and computer algebra is interesting in its own right. It also helps to compute higher dimensional Euler stratifications via fibration arguments like Lemma \ref{lem:hirz}. Based on computational data, we conjecture that the ideal of the stratum on which a binary $n$-form has two roots of multiplicity two, with partition $2^21^{n-4}$, is generated by $3n-5$ forms of degree $2n-5$.
    \item Theorem \ref{thm:tevelev} is stated for smooth varieties of dimension $d$, but used only for curves. Can one find an effective algorithm to compute the Euler stratification of hyperplane sections of a smooth surface via Theorem \ref{thm:tevelev}? One could start with smooth toric surfaces.
    \item Step $4$ in Algorithm \ref{alg:polardisc} and Algorithm \ref{alg:eulerdiscveryaff} can be performed using numerical elimination, along the lines of the numerical algorithms in \cite{fevola2024principal,mizera2022landau}. The development and implementation of such numerical algorithms is crucial for the application of Euler stratifications to more challenging Feynman integrals or statistical models. 
    \item The closure of the incidence variety ${\cal Y}_\nu^\circ$ in \eqref{eq:Ynucirc} is taken in $\mathbb{P}^d \times Z$. Here $\mathbb{P}^d$ can be replaced by any compactification of $T \simeq (\mathbb{C}^*)^d$. For instance, instead of projective space, one could consider the projective toric variety associated to the Newton polytope of $F(x,z)$ for generic $z$. We expect this to facilitate the closure computation.
    \item It is natural to study algorithms for computing the Euler stratification of families of varieties of higher codimension. 
    \item In the spirit of the discussion after Proposition \ref{prop:strictcontained}, it would be interesting to understand the projection of $\overline{{\cal Y}^\circ} \cap V(F) \times \mathbb{P}_\nu^d$. This requires to extend the work of Sattelberger and van der Veer~\cite{sattelberger2023maximum}, and to drop some assumptions made in that paper. 
    \item It is conjectured in \cite[p.~3]{clarke2023matroid} that for the GKZ family of bilinear forms given by \eqref{eq:fbilinear}, the signed Euler characteristic can take any value between $1$ and its generic value. As explained in Section \ref{sec:62}, this is a conjecture about realizable matroid invariants. 
    \item What is the precise relation between the Euler discriminant and the Landau variety from \cite[Definition 54]{brown2009periods} or \cite[Definition 1]{helmer2024landau}? Furthermore, given the fact that Feynman integrals satisfy nice differential equations in the parameters $z$ \cite[Section 4]{fourlectures}, how do these varieties relate to the singular locus of the annihilating $D$-module? Can this be used to understand whether there exist families of the form ${\cal X}_{f}^* \rightarrow Z$, see \eqref{eq:Xveryaff2}, for which the Euler discriminant $\nabla_\chi$ has codimension-two components in $Z$? 

    This question was addressed after the publication of a preprint of this article in \cite{helmer2025spectral,matsubara2025hypergeometric}.
\end{enumerate}

\section*{Acknowledgements}
We thank Alexander Esterov for many useful discussions about singularity theory and multisingularity strata, Martin Helmer for answering our questions about Whitney stratifications, and Giorgio Ottaviani for pointing us to Theorem \ref{thm:Dimca-Papadima}. Thanks to Eliana Duarte, Serkan Ho\c{s}ten, Saiei-Jaeyeong Matsubara-Heo and Bernd Sturmfels for encouraging conversations. 

\bibliographystyle{abbrv}
\small
\bibliography{references.bib}

\normalsize 
\noindent{\bf Authors' addresses:}
\medskip

\noindent Simon Telen, MPI MiS Leipzig
\hfill {\tt simon.telen@mis.mpg.de}

\noindent Maximilian Wiesmann, CSBD Dresden
\hfill {\tt wiesmann@pks.mpg.de}

\end{document}